\documentclass[a4paper,12pt]{amsart}
\usepackage{amsfonts}
\usepackage{color}
\usepackage{url}
\usepackage{bigints}
\usepackage[normalem]{ulem}

\usepackage{graphicx}
\usepackage{amsmath,amssymb,amsthm,amsfonts}
\usepackage{amssymb}
\usepackage[active]{srcltx}
\usepackage[colorlinks=true]{hyperref}
\usepackage{graphics}
\usepackage{graphicx}

\newtheorem{thm}{Theorem}[section]

\newtheorem{lem}[thm]{Lemma}
\newtheorem{prop}[thm]{Proposition}
\newtheorem{rem}[thm]{Remark}
\newtheorem{defn}[thm]{Definition}
\newcommand{\R}{\mathbb{R}}
\newcommand{\N}{\mathbb{N}}
\renewcommand{\S}{\mathbb{S}}


\begin{document}

\parindent 0pc
\parskip 6pt
\overfullrule=0pt

\title{On the Gibbons'  conjecture for equations involving the $p$-Laplacian}

\author{Francesco Esposito*, Alberto Farina$^+$, Luigi Montoro* and Berardino Sciunzi*}

\thanks{F. Esposito and L. Montoro and B. Sciunzi were partially supported by PRIN project  2017JPCAPN (Italy): {\em Qualitative and quantitative aspects of nonlinear PDEs}. F. Esposito and A. Farina were partially supported by PICS 2018 - VALABLE}

\address{*Dipartimento di Matematica e Informatica, UNICAL,
Ponte Pietro  Bucci 31B, 87036 Arcavacata di Rende, Cosenza, Italy.}

\email{esposito@mat.unical.it}

\email{montoro@mat.unical.it}

\email{sciunzi@mat.unical.it}

\address{+ Universit\'e de Picardie Jules Verne, LAMFA, CNRS UMR
7352, 33, rue Saint-Leu 80039 Amiens, France}

\email{alberto.farina@u-picardie.fr}


\begin{abstract}
In this paper we prove the validity of Gibbons' conjecture for the quasilinear elliptic equation $ -\Delta_p u = f(u) $ on $\R^N.$ The result holds true for $(2N+2)/(N+2) < p < 2$ and for a very general class of nonlinearity $f$.
\end{abstract}

\maketitle

\section{Introduction}
In this work we are  concerned with the study of qualitative properties of weak solutions of class $C^1$ to the quasilinear elliptic equation
\begin{equation}\label{equation1}\tag{$\mathcal P$}
-\Delta_p u = f(u) \quad \text{in} \,\, \R^N,
\end{equation}
where we denote a generic point of $\mathbb{R}^N$ by $(x',y)$ with $x'=(x_1,x_2, \ldots, x_{N-1}) \in \R^{N-1}$ and $y=x_N \in \R$, $p>1$ 
and $N > 1$. 
The nonlinear function $f$ will be assumed to satisfy the following  assumptions : 
\begin{equation*}
{(h_f)} \qquad \begin{cases}
f \in C^1([-1,1]), \quad f(-1)= f(1)=0, & \\
f_+'(-1)<0,  \quad f'_-(1)<0, & \\ 
\mathcal{N}_f:=\{t \in [-1,1] \ | \ f(t)=0\} \,\, {\text {is a finite set.}} & 
\end{cases}
\end{equation*}
A very special case covered by our assumptions is the well-known semilinear Allen-Cahn equation
\begin{equation}\label{eq:degiorgibb}
-\Delta u =u(1-u^2)\quad \text{in}\,\, \mathbb R^N,
\end{equation}
for which the following conjecture have been stated



{\sc Gibbons' conjecture \cite{Car}} {\textit {Assume $N>1$ and consider a bounded solution $u \in C^2(\mathbb R^N)$ 
of \eqref{eq:degiorgibb} such that
\[\lim_{x_N\rightarrow \pm \infty}u(x',x_N)=\pm 1,\]
uniformly with respect to $x'$. Then,  is it  true that
\[u(x)=\tanh \left(\frac{x_N-\alpha}{\sqrt 2} \right),\]
for some $\alpha\in \mathbb R$?}}


Gibbons' conjecture was proven independently and with different methods by \cite{BBG, BHM, Far99, FarLincei} (see also \cite{ FarDisc, Fa} for further results in the semilinear scalar case and \cite{FarSciuSo} for a recent result concerning some related semilinear elliptic systems).
Here we study Gibbons' conjecture for the quasilinear equation \eqref{equation1}. To the best of our knowledge, there are no general results in this framework. This lack of results is mainly due to the fact that, unlike the  semilinear case, when working with the singular operator 
$-\Delta_p(\cdot)$, both the weak and the strong comparison principles might fail. This (possible) failure being caused either by the presence of critical points or by the fact that the nonlinearity $f$ changes sign. Those difficulties are even more magnified by the fact that we are facing a problem on an unbounded domain, the entire euclidean space $\R^N.$ Also, in the pure quasilinear case, $p\neq 2$, we cannot exploit the usual arguments and tricks related to the linearity of the Laplace operator. Despite all those problems and difficulties, we are able to study and solve the quasilinear version of Gibbons' conjecture by making use of the the celebrated moving planes method which goes back to the papers of Alexandrov \cite{A} and Serrin  \cite{serrin} (see  also \cite{BN, GNN}).

\noindent Our main result is the following
\begin{thm}\label{thm:gibb}
\label{1Dsolution}
Assume $N > 1$,  ${(2N+2)}/{(N+2)} < p < 2$ and let $u \in C^1(\R^N)$ be a weak solution of
\eqref{equation1}, such that
\[|u|\leq1 \qquad  {\text {on}}\quad \R^N\]
and
\begin{equation}\label{inftyassumptions}
\lim_{y \rightarrow + \infty} u(x',y) = 1 \quad \text{and} \quad
\lim_{y \rightarrow - \infty} u(x',y) =-1,
\end{equation}
uniformly with respect to $x' \in \R^{N-1}$. If $f$ fulfills  {\bf ($h_f$)}, then $u$ depends only on $y$ and
\begin{equation}\label{monotonicityxn}
\partial_{y} u > 0 \quad \text{in} \quad \R^N.
\end{equation}
\end{thm}


To get our main result, we  first prove a new weak comparison principle for quasilinear equations in half-spaces and then we exploit it to start the moving plane procedure {at infinity in the $y$-direction. Then, by a delicate analysis based on the the use of the techniques developed in \cite{DSCalcVar, DSJDE} and \cite{FMS, FMS3, FMSR}, the translation invariance of the considered problem and the method introduced in \cite{Far99}, we obtain} the monotonicity of the solution in all the directions of the  the upper hemi-sphere $\S^{N-1}_+:=\{\nu \in \S^{N-1}_+ \ | \ (\nu, e_N) \}$.  {This result, in turn, will provide the desired one-dimensional symmetry result as well as the strict monotonicity.}

The paper is organized as follows: In Section \ref{sec: StrongMaxCompPrin} we recall {the definition of weak solution of \eqref{equation1}, as well as some results about the strong maximum principle and the comparison principles for nonlinear equations involving the $p$-Laplace operator.} In Section \ref{sec: pre} we prove a new weak comparison principle in half-spaces. In Section \ref{sec: moving} we prove the monotonicity of the solution in the $y$-direction, exploiting the moving plane procedure. In Section \ref{sec: 1sim} we prove the the one-dimensional symmetry and the strict monotonicity of the solution. 
	


\section{Strong maximum principles and strong comparison principles for quasilinear elliptic equations} \label{sec: StrongMaxCompPrin}
 {The aim of this section is to recall some results about the strong comparison principles and the strong maximum principles for quasilinear elliptic equations that will be used several times in the proof of our main theorem. To this end we first recall the definiton of weak solution for the quasilinear equation $-\Delta_p u = f(u)$.  }

	\begin{defn}\label{weakform}
		Let $\Omega$ be an open set of $\R^N$, $N\geq 1$. We say that $u \in  C^1(\Omega)$ is a
		\emph{weak subsolution} to
		\begin{equation}\label{eq:probleadjunt}
		-\Delta_p u = f(u) \quad \text{in} \quad\Omega
		\end{equation}
		if
		\begin{equation}\label{problem1weaksubsol}
		\int_{\Omega} |\nabla u|^{p-2}(\nabla u, \nabla \varphi) \, dx \, \leq
		\int_{\Omega} f(u) \varphi \, dx \qquad \forall \varphi \in
		C_c^\infty(\Omega), \, \varphi \geq 0.
		\end{equation}
		Similarly, we say that $u \in C^1(\Omega)$ is a
		\emph{weak supersolution} to \eqref{eq:probleadjunt} if
		\begin{equation}\label{problem1weaksupersol}
		\int_{\Omega} |\nabla u|^{p-2} (\nabla u, \nabla \varphi) \, dx \,
		\geq \int_{\Omega} f(u) \varphi \, dx \qquad \forall \varphi \in
		C_c^\infty(\Omega), \, \varphi \geq 0.
		\end{equation}
		Finally, we say that $u \in  C^1(\Omega)$ is a
		weak \emph{solution} of equation \eqref{eq:probleadjunt}, if
		\eqref{problem1weaksubsol} and \eqref{problem1weaksupersol}  hold.
		 {Sometimes for brevity,  we shall use the term "solution" 
		to indicate a weak solution to the considered problem. }
	\end{defn}


The first result that we are going to present is the classical strong maximum principle due to J. L. Vazquez \cite{Vaz} (see also P. Pucci and J. Serrin book \cite{pucser}) 

\begin{thm}[Strong Maximum Principle and H\"opf's Lemma, \cite{pucser, Vaz}]\label{StrongMaximumPrinciple}
Let  $u \in C^1(\Omega)$ be a non-negative weak solution to
$$-\Delta_p u + c u^q = g \geq 0 \quad \text{in} \quad \Omega$$
with $1<p<+\infty$, $q\geq p-1$, $c \geq 0$ and $g \in L^\infty_{loc} (\Omega)$. If $u \neq 0$, then $u>0$ in $\Omega$.
Moreover for any point $x_0 \in \partial \Omega$ where the interior sphere condition is satisfied, and such that $u \in C^1(\Omega) \cup \{x_0\}$  and $u(x_0) = 0$ we have that $\partial_\nu u>0$ for any inward directional derivative (this means that if $y$ approaches $x_0$ in a ball $B \subseteq \Omega$ that has $x_0$ on its boundary, then $\lim_{y \rightarrow x_0} \frac{u(y)-u(x_0)}{|y-x_0|}>0$).
\end{thm}

It is very simple to guess that in the quasilinear case, maximum and comparison principles are not equivalent; for this reason we need also to recall the classical version of the strong comparison principle for quasilinear elliptic equations 
\begin{thm}[Classical Strong Comparison Principle, \cite{Damascelli, pucser}] \label{classicalSCP}

\noindent Let $u,v \in C^1(\Omega)$ be two solutions to 
\begin{equation} \label{probSMCP}
-\Delta_p w  = f(w) \qquad  \text{in} \quad \Omega
\end{equation}
such that $u \leq v$ in $\Omega$, with $1<p<+\infty$ and let $\mathcal{Z} =\{x \in \Omega \ | \ |\nabla u (x)| + |\nabla v (x)|  =  0\}$.	If $x_0 \in \Omega \setminus \mathcal{Z} $ and $u(x_0)=v(x_0)$, then $u=v$ in the connected component of $\Omega \setminus \mathcal{Z}$ containing $x_0$.
\end{thm}

 {For the proof of this result we suggest \cite{Damascelli}. } The main feature of Theorem \ref{classicalSCP} is that it holds far from the critical set. Now we present a result which holds true, under stronger assumptions,  {on the entire domain $\Omega$.}   

\begin{thm}[Strong Comparison Principle, \cite{DSCalcVar}]\label{SCPLucioeDino}
	Let $u,v\in C^1(\overline{\Omega})$ be two solutions to \eqref{probSMCP}, where $\Omega$ is a bounded  domain 
	of $\mathbb{R}^N$ and $\frac{2N+2}{N+2}<p<+\infty$.
Assume that at least one of the following two conditions \textbf{$(f_u)$},\textbf{$(f_v)$} holds:
\begin{itemize}
	\item[\textbf{($f_u$):}] either
	\begin{equation}\label{fuasspos}
	f(u(x)) >0 \quad\mbox{in}\quad\overline{\Omega}
	\end{equation}
	or
	\begin{equation}\label{fuassneg}
	f(u(x)) <0 \quad\mbox{in}\quad\overline{\Omega};
	\end{equation}

	\item[\textbf{($f_v$):}] either
	\begin{equation}\label{fvasspos}
	f(v(x)) >0 \quad\mbox{in}\quad\overline{\Omega}
	\end{equation}
	or
	\begin{equation}\label{fvassneg}
	f(v(x)) <0 \quad\mbox{in}\quad\overline{\Omega}.
	\end{equation}
\end{itemize}
 {Suppose furthermore that}
\begin{equation}\label{hthCOMP:LAMB22II}
u\leq v\quad\mbox{in}\quad\Omega.
\end{equation}
Then $u\equiv v$ in $\Omega$ unless
\begin{equation}
u<v\quad \mbox{in}\quad\Omega.
\end{equation}
\end{thm}
\begin{proof}
The proof of this result follows by the same arguments in \cite{DSCalcVar, FMS3, SciunziNodea, SciunziCCM}. Note in fact that under the assumption \textbf{($f_u$)} or \textbf{($f_v$)}, it follows that $|\nabla u|^{-1}$ or $|\nabla v|^{-1}$ has the summability properties exposed by Theorem 3.1 in \cite{SciunziCCM}. Then the weighted Sobolev inequality is in force, see e.g. Theorem 8 in \cite{FMS3}.

Now, it is sufficient to note that the Harnack comparison inequality given by Corollary 3.2 in \cite{DSCalcVar} holds true, since the proof it is only based on the weighted Sobolev inequality.

Finally it is standard to see that the Strong Comparison Principle follows by the weak comparison Harnack inequality (that it is based on the Moser-iteration scheme \cite{moser, trudinger}), see Theorem 1.4 in \cite{DSCalcVar}.
\end{proof}	

\smallskip

Let us now recall that the linearized operator at a fixed solution $w$ of \eqref{probSMCP}, $L_w(v,\varphi)$, is well defined, for every $v$ and $\varphi$ in the weighted Sobolev space $H^{1,2}_\rho(\Omega)$ with $\rho=|\nabla w|^{p-2}$ by
\begin{equation}\label{eq:linearizzatoSCP}
\begin{split}
L_w(v, \varphi) \equiv \int_\Omega |\nabla w|^{p-2}(\nabla v, \nabla \varphi)+(p-2) |\nabla w|^{p-4}(\nabla w, \nabla v)(\nabla w, \nabla \varphi) - f'(w)v \varphi \, dx .
\end{split}
\end{equation}
Moreover $v \in H^{1,2}_\rho(\Omega)$ is a weak solution of the linearized operator if
\begin{equation} \label{linearizedequation}
L_w(v, \varphi)=0 \qquad {\forall \varphi \in H^{1,2}_{0,\rho}(\Omega)}.
\end{equation}

For future use we recall that, as it follows by the regularity results in \cite{DSJDE, SciunziNodea, SciunziCCM}, the directional derivatives of the solution $\partial_{\eta} u$ ($\eta \in \S^{N-1}$) belong to the weighted Sobolev space $H^{1,2}_{\rho}(\Omega)$ and fulfils \eqref{linearizedequation}.

In particular here below we recall two versions of the strong maximum principle for the linearized equation \eqref{linearizedequation} that we shall use in our proofs. The first result holds far from the critical set:

\begin{thm}[Classical Strong Maximum Principle for the Linearized Operator, \cite{pucser}] \label{ClassicSMPlinearized}
Let $u\in C^1(\overline{\Omega})$ be a solution to problem \eqref{probSMCP}, with $1<p<+\infty$. Let $\eta \in \S^{N-1}$ and let us assume that for any connected domain $\Omega ' \subset \Omega \, \setminus $  {$ \{x \in \Omega \ | \ |\nabla u (x)|  =  0\}$. }
\begin{equation}\label{linass}
\partial_\eta u\geq 0 \quad  \text{in} \quad \Omega'.
\end{equation}
Then $\partial_\eta u \equiv 0$ in $\Omega '$ unless
\begin{equation}\label{linthesis}
\partial_\eta u> 0 \quad  \text{in} \quad \Omega'.
\end{equation}
\end{thm}

{Next we recall a more general result which holds true on the entire domain $\Omega$.}   

\begin{thm}[Strong Maximum Principle for the Linearized Operator, \cite{DSCalcVar}]  \label{SMPlinearized}
	Let $u\in C^1(\Omega)$ be a solution to problem \eqref{probSMCP}, with $\frac{2N+2}{N+2}<p<+\infty$. Assume that either
\begin{equation}\label{fassposbis}
f(u(x)) >0 \quad\mbox{in}\quad\overline{\Omega}
\end{equation}
or
\begin{equation}\label{fassnegbis}
f(u(x)) <0 \quad\mbox{in}\quad\overline{\Omega}.
\end{equation}
If $\eta \in \S^{N-1}$ and $\partial_\eta u \geq 0$ in $\Omega$, then  either $\partial_\eta u \equiv 0$ in $\Omega$ or $\partial_\eta u>0$ in $\Omega$.
\end{thm}

\

{We conclude this section by the following}

\begin{rem}\label{rem:1}
	{We want to point out the following properties satisfied by any weak solution to \eqref{equation1} such that $ \vert u \vert \leq 1$ on $\R^N$. They will be used several times throughout the paper.
	\begin{itemize}
		\item By the strong maximum principle \cite{Vaz}, see also Theorem \ref{StrongMaximumPrinciple}, we deduce that: either $|u|<1$ on $\R^N$ or $ u\equiv\pm 1$ on $\R^N$.  
	    \item By classical regularity results \cite{DB,T} and since $\|f(u)\|_{L^{\infty}(\mathbb R^N)} \leq \Vert f \Vert_{L^{\infty}([-1,1])}$, 
	    we deduce that : given $R \in (0,1)$ there exist $\alpha \in (0,1)$ and $C>0$, depending only on $N$, $p$, $R$ and 
	    $\Vert f \Vert_{L^{\infty}([-1,1])}$, so that 
	    \[ 
	    \|\nabla u\|_{L^{\infty}(\mathbb R^N)}\leq C,
	    \]
	    \[
	    \vert \nabla u(x) - \nabla u(y) \vert \leq C \Big( \frac{\vert x-y\vert}{R}\Big)^{\alpha},
	    \]
	    for every $x_0\in \R^N$ and any $x,y \in B_R(x_0).$ 
	    In particular, $u \in C^{1,\alpha}_{loc}(\R^N).$
	    \end{itemize} }
\end{rem}

\section{Preliminary results}\label{sec: pre}

{In this section we shall denote by $\Sigma$ any (affine) open half-space of $ \R^N$ of the form 
\[\Sigma:=  \mathbb{R}^{N-1}\times (a,b),\]
where either $a=-\infty$ and $b\in \mathbb R$, or  $a\in \mathbb R$ and $b=+\infty$.}

{We also recall some known inequalities which will be used in this section.} For any $\eta, \eta' \in \mathbb{R}^N$ with $|\eta|+|\eta'|>0$ there
exists positive constants $C_1, C_2, C_3$ {depending only on} $p$ such that
\begin{equation}\label{eq:inequalities}
\begin{split}
[|\eta|^{p-2}\eta-|\eta'|^{p-2}\eta'][\eta- \eta'] &\geq C_1
(|\eta|+|\eta'|)^{p-2}|\eta-\eta'|^2, \\ \\
\|\eta|^{p-2}\eta-|\eta'|^{p-2}\eta '| & \leq C_2
(|\eta|+|\eta'|)^{p-2}|\eta-\eta '|,\\\\
\|\eta|^{p-2}\eta-|\eta'|^{p-2}\eta '| & \leq C_3 |\eta-\eta
'|^{p-1} \quad\mbox{if}\quad 1 < p \leq 2.
\end{split}
\end{equation}

\medskip

The first result that we need is a weak comparison principle between a subsolution and a supersolution to \eqref{equation1} ordered on the boundary of some open half-space $\Sigma$ of $\R^N$. We prove the following
\begin{prop}\label{weakcomparisonprinciple}
{Assume $N>1, \, p>1$ and $f\in C^1(\mathbb R)$. Let $u, v \in C^{1, \alpha}_{loc}(\overline{\Sigma})$ such that $|\nabla u|, |\nabla v| \in L^\infty ({\Sigma})$} and 
\begin{equation}\label{supersolprob}
\begin{cases} \displaystyle
-\Delta_p u \leq f(u) & \text{in}\,\, \Sigma,\\
\displaystyle-\Delta_p v  \geq f(v) & \text{in}\,\, \Sigma,\\
\displaystyle u \leq v & \text{on}\,\,\partial \Sigma,
\end{cases}
\end{equation} 
{where  $\Sigma$ is the open half-space $\mathbb{R}^{N-1}\times (-\infty,b).$ 
Moreover, let us assume that there are $\delta>0$, sufficiently small, and $L>0$ such that }
\begin{equation}\label{eq:derivatafnegativa}
f'(t)<- L\quad  \text{in} \quad [-1,-1+\delta],
\end{equation}
\begin{equation}\label{eq:vicinoamenouno}
-1 \leq u \leq -1 + \delta \quad  \text{in} \quad \Sigma.
\end{equation}
Then
\begin{equation}\label{weakcomp}
u \leq v \quad \text{in} \,\,\Sigma.
\end{equation}

The same result is true if {$\Sigma = \mathbb{R}^{N-1}\times (a, +\infty)$ and \eqref{eq:derivatafnegativa} and \eqref{eq:vicinoamenouno} are replaced by }
\[
 f'(t)<- L\quad  \text{in} \quad [1-\delta,1] \quad \text{and} \quad  1-\delta \leq v \leq 1 \quad \text{in} \quad  \Sigma. 
 \]
\end{prop}

\begin{proof} We prove the result  {when \eqref{eq:derivatafnegativa} and \eqref{eq:vicinoamenouno} are in force. The other case is similar.}


We distinguish two cases:

\noindent {\bf {Case 1: $1<p <2$.}} We set
\begin{equation}\label{test}
\psi:=w^\alpha \varphi_R^{\alpha+1},
\end{equation}
where $\alpha >1$,  $R >0$ large,
$w:=(u-v)^+$ and $\varphi_R$ is a standard cutoff function such that
$0\leq \varphi_R \leq 1$ on $\R^N$, $\varphi_R = 1$ in $B_R$,
$\varphi_R = 0$ outside $B_{2R}$, with $|\nabla \varphi_R| \leq 2/R$
in $B_{2R} \setminus B_R$. Let us define $\mathcal{C}(2R):=\Sigma \cap B_{2R} \cap \text{supp} (\omega)$. First of all we notice that $\psi\in W^{1,p}_0(\mathcal{C}(2R))$. By density arguments we can take $\psi$
as test function in \eqref{problem1weaksubsol} and
\eqref{problem1weaksupersol}, so that, subtracting we obtain
\begin{equation}\label{starteqpminhi}
\begin{split}
&\alpha \int_{\mathcal{C}(2R)} (|\nabla u|^{p-2} \nabla u - |\nabla
v|^{p-2} \nabla v, \nabla w)
w^{\alpha-1} \varphi_R^{\alpha+1} \, dx \\
&\leq - (\alpha+1) \int_{\mathcal{C}(2R)}  (|\nabla u|^{p-2} \nabla
u -
|\nabla v|^{p-2} \nabla v, \nabla \varphi_R) w^\alpha \varphi_R^{\alpha+1} \, dx\\
& \ \ \ \ + \int_{\mathcal{C}(2R)}  [f(u)-f(v)] w^\alpha
\varphi_R^{\alpha+1} \, dx\,.
\end{split}
\end{equation}
From \eqref{starteqpminhi}, using \eqref{eq:inequalities} and
noticing that $f$ is decreasing in $[-1,-1+{\delta}]$, we obtain
\begin{equation}\label{middlediseq1}
\begin{split}
&\alpha C_1 \int_{\mathcal{C}(2R)}  \left(|\nabla u| + |\nabla
v|\right)^{p-2} |\nabla w|^2
w^{\alpha-1} \varphi_R^{\alpha+1} \, dx \\
&\leq \alpha \int_{\mathcal{C}(2R)} (|\nabla u|^{p-2} \nabla u -
|\nabla v|^{p-2} \nabla v, \nabla w)
w^{\alpha-1} \varphi_R^{\alpha+1} \, dx \\
&\leq - (\alpha+1) \int_{\mathcal{C}(2R)}  (|\nabla u|^{p-2} \nabla
u -
|\nabla v|^{p-2} \nabla v, \nabla \varphi_R) w^\alpha \varphi_R^\alpha \, dx\\
& \ \ \ \  + \int_{\mathcal{C}(2R)}  f'(\xi) (u-v)^+w^\alpha
\varphi_R^{\alpha+1} \,
dx\,\\
&\leq (\alpha+1) C_3 \int_{\mathcal{C}(2R)} |\nabla w|^{p-1} |\nabla
\varphi_R| w^\alpha \varphi_R^\alpha \, dx - L\int_{\mathcal{C}(2R)}
(u- v)^+w^\alpha \varphi_R^{\alpha+1} \, dx,
\end{split}
\end{equation}
where $\xi$ is some point that belongs to $(v,u)$.
Hence, recalling also that $|\nabla u|, |\nabla v| \in
L^\infty({\Sigma})$, we deduce
\begin{equation}\label{middlediseq2}
\begin{split}
&\alpha C_1 \int_{\mathcal{C}(2R)}  \left(|\nabla u| + |\nabla
v|\right)^{p-2} |\nabla w|^2
w^{\alpha-1} \varphi_R^{\alpha+1} \, dx\\
&\leq (\alpha+1) C_3 \int_{\mathcal{C}(2R)} |\nabla w|^{p-1} |\nabla
\varphi_R| w^\alpha \varphi_R^\alpha \, dx -L\int_{\mathcal{C}(2R)}
w^{\alpha+1} \varphi_R^{\alpha+1}
\, dx \\
& \leq (\alpha+1) C \int_{\mathcal{C}(2R)} |\nabla
\varphi_R| w^\alpha \varphi_R^\alpha \, dx - L\int_{\mathcal{C}(2R)}
w^{\alpha+1} \varphi_R^{\alpha+1} \, dx
\end{split}
\end{equation}
where $C=C(p, \|\nabla
u\|_{L^\infty({\Sigma})},\|\nabla
v\|_{L^\infty({\Sigma})})$. Exploiting  the weighted
Young inequality with exponents $\alpha+1$ and ${(\alpha+1)}/{\alpha}$ in \eqref{middlediseq2},  we obtain
\begin{equation}\nonumber
\begin{split}
&\alpha C_1 \int_{\mathcal{C}(2R)}  \left(|\nabla u| + |\nabla
v|\right)^{p-2} |\nabla w|^2
w^{\alpha-1} \varphi_R^{\alpha+1} \, dx\\
&\leq \frac{C}{\sigma^{\alpha+1}}
\int_{\mathcal{C}(2R)} |\nabla \varphi_R|^{\alpha+1} \, dx+
 \alpha C
\sigma^{\frac{\alpha+1}{\alpha}}\int_{\mathcal{C}(2R)} w^{\alpha+1} \varphi_R^{\alpha+1}\, dx
 \\
 &-L\int_{\mathcal{C}(2R)}
w^{\alpha+1} \varphi_R^{\alpha+1} \, dx
\\
&\leq \frac{C}{\sigma^{\alpha+1}}
\int_{\mathcal{C}(2R)} |\nabla \varphi_R|^{\alpha+1} \, dx+ \left( \alpha C\sigma^{\frac{\alpha+1}{\alpha}} - L\right)
\int_{\mathcal{C}(2R)} w^{\alpha+1} \varphi_R^{\alpha+1}\, dx\\
&\leq
\frac{2^{\alpha+1}C}{\sigma^{{\alpha+1}}R^{\alpha-(N-1)}}
+ \left(\alpha C\sigma^{\frac{\alpha+1}{\alpha}} -L
\right) \int_{\mathcal{C}(2R)} w^{\alpha+1} \varphi_R^{\alpha+1} \,
dx.
\end{split}
\end{equation}
Now taking $\alpha> N-1$, if we choose $\sigma=\sigma(p,\alpha, L,N, \|\nabla
u\|_{L^\infty({\Sigma})},\|\nabla
v\|_{L^\infty({\Sigma})})>0$ sufficiently small
so that
\[\alpha C\sigma^{\frac{\alpha+1}{\alpha}} -L
< 0,\] we obtain
\begin{equation}\label{finaldiseq}
\int_{\mathcal{C}(R)} \left(|\nabla u| + |\nabla v|\right)^{p-2}
|\nabla w|^2 w^{\alpha-1} \, dx  \leq \frac{\tilde{C}}{\alpha\sigma^{{\alpha+1}}R^{\alpha-(N-1)}}.
\end{equation}
Passing to the limit in \eqref{finaldiseq} for $R \rightarrow +
\infty$, by Fatou's Lemma  we have
$$\int_{\Sigma} \left(|\nabla u| + |\nabla v|\right)^{p-2}
|\nabla w|^2 w^{\alpha-1} \, dx  \leq 0.$$ This implies that $u \leq
v$ in $\Sigma$.

\noindent {\bf {Case 2: $p \geq 2$.}} We set
\begin{equation}\label{test2}
\psi:=w \varphi_R^2,
\end{equation}
where $R >0$, $w:=(u-v)^+$
and $\varphi_R$ is the standard cutoff function defined above. First
of all we notice that $\psi\in W^{1,p}_0(B_{2R})$. Let us define $\mathcal{C}(2R):=\Sigma \cap B_{2R} \cap \text{supp} (\omega)$. By
density arguments we can take $\psi$ as test function in
\eqref{problem1weaksubsol} and \eqref{problem1weaksupersol}, so
that, subtracting we obtain
\begin{equation}\label{starteqpminhi2}
\begin{split}
&\int_{\mathcal{C}(2R)} (|\nabla u|^{p-2} \nabla u - |\nabla
v|^{p-2} \nabla v, \nabla w)
 \varphi_R^2 \, dx \\
&\leq- 2 \int_{\mathcal{C}(2R)}  (|\nabla u|^{p-2} \nabla u -
|\nabla v|^{p-2} \nabla v, \nabla \varphi_R) w \varphi_R \, dx\\
& \ \ \ \ + \int_{\mathcal{C}(2R)}  [f(u)-f(v)] w \varphi_R^2 \, dx\,.
\end{split}
\end{equation}
From \eqref{starteqpminhi2}, using \eqref{eq:inequalities} and that $f'(u) \leq -L$  in $[-1,-1+\delta]$, we obtain
\begin{equation}\label{middlediseq3}
\begin{split}
& C_1 \int_{\mathcal{C}(2R)}  \left(|\nabla u| + |\nabla
v|\right)^{p-2} |\nabla w|^2 \varphi_R^2 \, dx \\
&\leq  \int_{\mathcal{C}(2R)} (|\nabla u|^{p-2} \nabla u - |\nabla
v|^{p-2} \nabla v, \nabla w) \varphi_R^2 \, dx \\
&\leq- 2 \int_{\mathcal{C}(2R)}  (|\nabla u|^{p-2} \nabla u -
|\nabla v|^{p-2} \nabla v, \nabla \varphi_R) w \varphi_R \, dx\\
&\ \ \  +\int_{\mathcal{C}(2R)} f'(\xi) (u-v)^+w
\varphi_R^2 \,
dx\,\\
&\leq 2 C_2 \int_{\mathcal{C}(2R)} \left(|\nabla u| + |\nabla
v|\right)^{p-2} |\nabla w| \ |\nabla \varphi_R| w \varphi_R \, dx
\\
&\ \ \ - L \int_{\mathcal{C}(2R)} (u- v)^+w \varphi_R^2 \, dx,
\end{split}
\end{equation}
where $\xi$ is some point that belongs to $(v,u)$. Using  in \eqref{middlediseq3} the weighted Young inequality (and the fact that $|\nabla
u|, |\nabla v| \in L^\infty({\Sigma})$), we obtain

\begin{equation}
\begin{split}
 C_1 \int_{\mathcal{C}(2R)}&  \left(|\nabla u| + |\nabla
v|\right)^{p-2} |\nabla w|^2 \varphi_R^2 \, dx \\
&\leq  2 C_2 \int_{\mathcal{C}(2R)} \left(|\nabla u| +
|\nabla v|\right)^{\frac{p-2}{2}} |\nabla w|
\left(|\nabla u| + |\nabla
v|\right)^{\frac{p-2}{2}} |\nabla \varphi_R| w \varphi_R  \, dx \\
&\ \ \ - L \int_{\mathcal{C}(2R)} w^2 \varphi_R^2 \,
dx\\
&\leq  C_2 \sigma \int_{\mathcal{C}(2R)} \left(|\nabla u| + |\nabla
v|\right)^{p-2} |\nabla w|^2 \, dx \\
& \ \ \ + \frac{C_2}{\sigma} \int_{\mathcal{C}(2R)} \left(|\nabla u| +
|\nabla v|\right)^{p-2} |\nabla
\varphi_R|^2 w^2 \varphi_R^2 \, dx \\
&\ \ \ - L \int_{\mathcal{C}(2R)} w^2 \varphi_R^2  \, dx.\\
&\leq  C_2 \sigma \int_{\mathcal{C}(2R)} \left(|\nabla u| + |\nabla
v|\right)^{p-2} |\nabla w|^2 \, dx \\
&\ \ \ + \left(\frac{C}{\sigma R^2} - L\right)
\int_{\mathcal{C}(2R)} w^2 \varphi_R^2 \, dx,
\end{split}
\end{equation}
where $C=C(p, \|\nabla
u\|_{L^\infty({\Sigma})},\|\nabla
v\|_{L^\infty({\Sigma})})$ is a positive constant. Hence, up to redefine the constants,  we have
\begin{equation}\label{fineldiseq}
\begin{split}
\int_{\mathcal{C}(R)}  \left(|\nabla u| + |\nabla v|\right)^{p-2}
|\nabla w|^2 \, dx &\leq C\sigma
\int_{\mathcal{C}(2R)} \left(|\nabla u| + |\nabla v|\right)^{p-2}
|\nabla w|^2 \, dx \\
&\ \ \ + \frac{1}{C_1}\left(\frac{C}{\sigma R^2} - L\right)
\int_{\mathcal{C}(2R)} w^2 \varphi_R^2 \, dx.
\end{split}
\end{equation}
Now we set
\[\displaystyle \mathcal{L}(R):=\int_{\mathcal{C}(R)}
\left(|\nabla u| + |\nabla v|\right)^{p-2} |\nabla w|^2 \, dx.\]
By our assumption,$|\nabla u|, |\nabla v| \in
L^\infty(\Sigma)$, it follows that $\mathcal{L}(R) \leq \dot{C} R^N$
for every $R>0$ and for some $\dot C=\dot C(p, \|\nabla
u\|_{L^\infty({\Sigma})},\|\nabla
v\|_{L^\infty({\Sigma})})$. Moreover, in equation \eqref{fineldiseq}, we take $\sigma=\sigma(p, N, \|\nabla
u\|_{L^\infty({\Sigma})},\|\nabla
v\|_{L^\infty({\Sigma})})>0$ sufficiently small so
that $C\sigma < {1}/{2^N}$. Finally we  fix  $R_0
>0$ such that \[\displaystyle \frac{C}{\sigma R^2} - L < 0\]
for every $R \geq R_0$. Therefore by \eqref{fineldiseq} we deduce that
\begin{equation}
\begin{cases}
\mathcal{L}(R) \leq \vartheta \mathcal{L}(2R) \quad &\forall R \geq
R_0\\
\mathcal{L}(R) \leq \dot{C} R^N \quad &\forall R \geq R_0,
\end{cases}
\end{equation}
where $\displaystyle \vartheta:=C\sigma<1/2^N$. By applying
Lemma 2.1 in \cite{FMS} it follows that $\mathcal{L}(R)=0$ for all
$R \geq R_0$. Hence $u \leq v$ in $\Sigma$.

\end{proof}

Let us recall a weak comparison principle in narrow domains that will be an essential tool in the proof of Theorem \ref{1Dsolution}.
\begin{thm}[\cite{FMSR}]\label{compprinciplenarrow}
Let $1<p<2$ and $N>1$. Fix $\lambda_0 > 0$ and $L_0>0$. Consider $a,b \in \R$, with $a < b$, $\tau, \epsilon > 0$ and set
$${\Sigma_{(a,b)}:= \R^{N-1} \times (a,b).}$$
Let $u,v \in C^{1,\alpha}_{loc}(\overline{\Sigma}_{(a,b)})$ such that $\|u\|_\infty + \|\nabla u\|_\infty \leq L_0$, $\|v\|_\infty + \|\nabla v\|_\infty \leq L_0$, $f$ fulfills \textbf{$(h_f)$} and
\begin{equation}
\begin{cases}
-\Delta_p u \leq f(u) \quad & \text{in} \  \Sigma_{(a,b)} \\
-\Delta_p v \geq f(v) \quad & \text{in} \  \Sigma_{(a,b)} \\
u \leq v \quad & \text{on} \  \partial \mathcal{S}_{(\tau, \epsilon)}, \\
\end{cases}
\end{equation}
where the open set $\mathcal{S}_{(\tau, \epsilon)} \subseteq \Sigma_{(a,b)}$ is such that
$$\mathcal{S}_{(\tau, \epsilon)} = \bigcup_{x' \in \R^{N-1}} I_{x'}^{\tau, \epsilon},$$
and the open set $I_{x'}^{\tau, \epsilon} \subseteq \{x'\} \times (a,b)$ has the form
$$I_{x'}^{\tau, \epsilon} = A_{x'}^\tau \cup B_{x'}^\epsilon, \; \text{with} \; |A_{x'}^\tau \cap B_{x'}^\epsilon| = \emptyset$$
and, for $x'$ fixed, $A_{x'}^\tau, B_{x'}^\epsilon \subset (a,b)$ are measurable sets such that
$$|A^\tau_{x'}| \leq \tau \quad \text{and} \quad B_{x'}^\epsilon \subseteq \{x_N \in \R \ | \ |\nabla u(x',x_N)| < \epsilon, \ |\nabla v (x',x_N)| < \epsilon \}.$$
Then there exist
$$\tau_0=\tau_0(N,p,a,b,L_0)>0$$
and
$$\epsilon_0=\epsilon_0(N,p,a,b,L_0)>0$$
such that, if $0 < \tau < \tau_0$ and $0 < \epsilon < \epsilon_0$, it follows that
$$u \leq v \quad \text{in} \,\,\mathcal{S}_{(\tau, \epsilon)}.$$
\end{thm}

The proof of this result is contained in \cite[Theorem 1.6]{FMSR}, where the authors proved the same result for a more general class of operators and nonlinearities and also in the presence of a first order term.

\section{Monotonicity with respect to $x_N$}\label{sec: moving}

The purpose of this section consists in showing that all the non-trivial solutions $u$  to \eqref{equation1} that satisfies \eqref{inftyassumptions} are increasing in the $x_N$ direction. Since in our problem the right hand side
depends only on $u$, it is possible to define the following set
$$\mathcal{Z}_{f(u)}:= \{ x \in \R^N \ | \ u(x) \in \mathcal{N}_f\}.$$
{Without any apriori assumption on the behaviour of $\nabla u$, the set $\mathcal{Z}_{f(u)}$ may be very wild, see Figure \ref{fg:wild}.}
\begin{figure}[htbp]
	\centering
	\includegraphics[scale=.3]{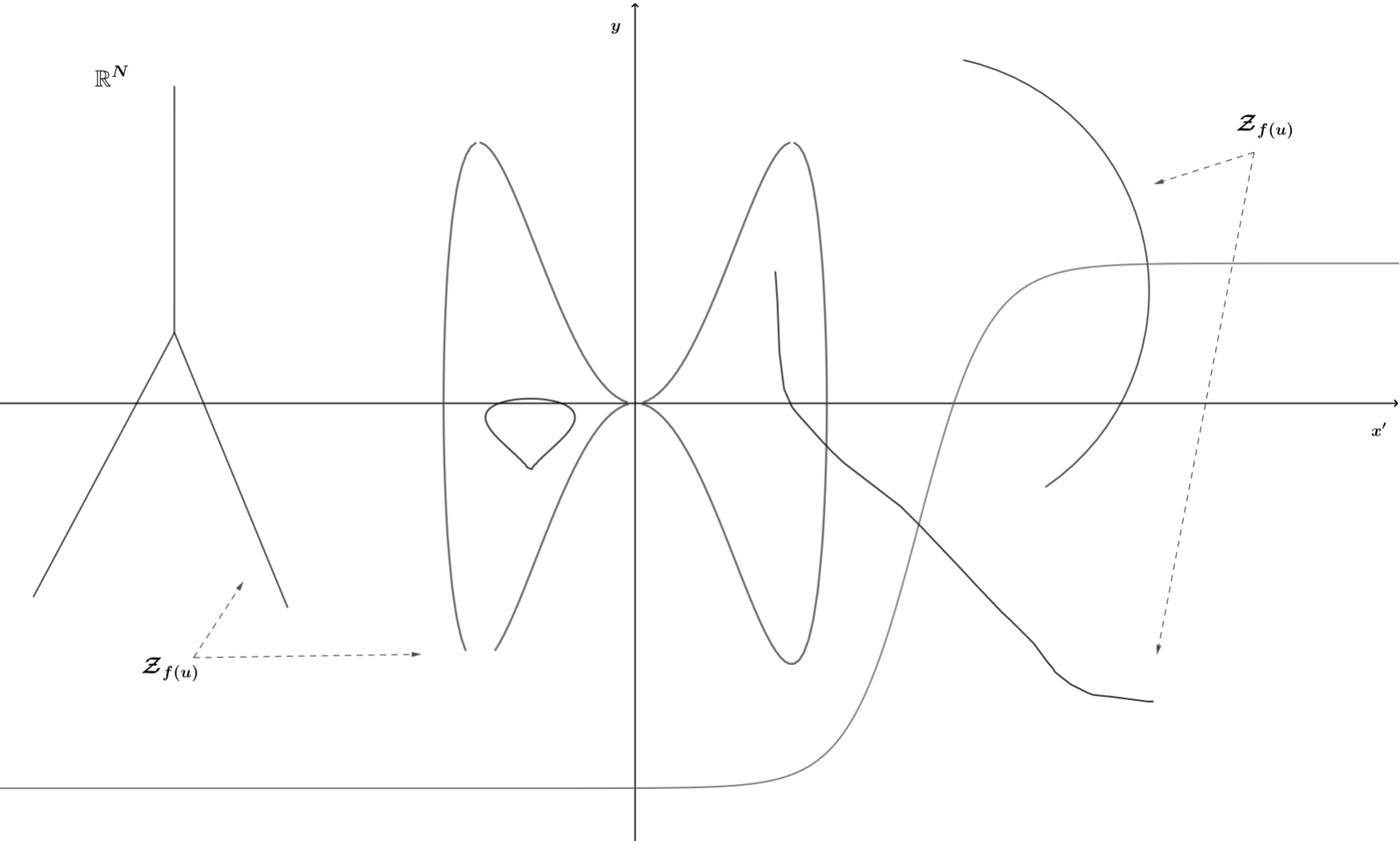}\\
	\caption{The  set $\mathcal{Z}_{f(u)}$ \label{fg:wild}}
\end{figure}

We start by proving a lemma that we will use repeatedly in the sequel of the work.

Let us define the upper hemisphere
\begin{equation}\label{eq:hemisphere}
\S^{N-1}_+:=\{\nu \in \S^{N-1} \ | \ (\nu, e_N)>0\}.
\end{equation}

\begin{lem}\label{lem:utile}
Let $\mathcal U$ a connected component of $\mathbb R^N\setminus \mathcal{Z}_{f(u)}$, $\eta \in \S^{N-1}_+$ and let us assume that $\partial_\eta u\geq 0$ in $\mathcal U$.
Then
\[\partial_\eta u> 0\quad \text{in}\,\,\mathcal U.\]
\end{lem}
\begin{proof}
Using Theorem \ref{SMPlinearized} we deduce that either $\partial_\eta u> 0$ in $\mathcal U$ or $\partial_\eta u\equiv 0$ in $\mathcal U$. 
{For contradiction, assume that} 
$\partial_\eta u \equiv 0$ in $\mathcal U$.  {Pick $P_0 \in \mathcal U$ and let us define}
\[r(t)=P_0+t\eta, \quad t\in \mathbb{R}\]
and 
\begin{equation}\label{eq:inffff}
t_0=\inf \Big\{t\in \mathbb R\,:\, r(\vartheta) \in \mathcal U,  \; \forall \vartheta \in (t,0]\Big\}.
\end{equation}
We note that the infimum in \eqref{eq:inffff} is well defined, since by definition the connected component $\mathcal U$ is an open set, and that 
$t_0 \in [-\infty, 0)$. 

In the case $t_0=-\infty$, we deduce that $u(P_0)=-1$. Indeed $u$ is constant on $r(t)$ for $t\in (-\infty, 0]$  (recall that  $\partial_\eta u \equiv 0$ in $\mathcal U$) and \eqref{inftyassumptions} holds. But this is a contradiction, see Remark \ref{rem:1}.

In the case $t_0>-\infty$, we deduce that $r(t_0)\in \mathcal{Z}_{f(u)}$ and therefore $f(u(r(t_0)))=f(u(P_0+t_0\eta))=0$. But $u$ is constant on $r(t)$ for $t_0\leq t\leq 0$, which implies that $f(u(P_0))=f(u(P_0+t_0\eta))=0$, namely $P_0\in \mathcal{Z}_{f(u)}$. {The latter clearly contradicts the assumption $P_0 \in \mathcal U$. Therefore $\partial_\eta u> 0$ in $\mathcal U$ as desired.}
\end{proof}

\begin{prop}\label{monotonicitybis}
Under the assumptions of Theorem \ref{1Dsolution}, we have that
\begin{equation}\label{mmmmmm}
\partial_{x_N} u > 0 \quad \text{in} \,\, \R^N\setminus
\mathcal{Z}_{f(u)}.
\end{equation}
\end{prop}
The proof is based on a nontrivial modification of the moving plane method. Let us recall some notations. We
define the half-space $\Sigma_\lambda$ and the hyperplane
$T_\lambda$ by
\begin{equation} \label{eq:acquaesapone} \begin{split}
\Sigma_\lambda := \{x \in \R^N \ | \ x_N < \lambda\},\ \quad  \quad
T_\lambda &:= \partial \Sigma_\lambda = \{x \in \R^N \ | \ x_N = \lambda\}
\end{split} \end{equation}
and the reflected function $u_\lambda(x)$ by
$$u_\lambda (x) = u_\lambda (x',x_N):= u(x', 2\lambda-x_N) \quad \text{in}\,\, \mathbb R^N.$$
We also define the critical set $\mathcal{Z}_{\nabla u}$ by
\begin{equation}\label{eq:criticalset}
\mathcal{Z}_{\nabla u}:=\{x \in \R^N \ | \ \nabla u(x) = 0 \}.
\end{equation}
The first step in the proof of the monotonicity is to get a property concerning the local
symmetry regions of the solution, namely any $C \subseteq \Sigma_\lambda$ such that $u \equiv u_\lambda$ in $C$.

Having in mind these notations we are able to prove the following:

\begin{prop}\label{connectedness}
Under the assumption of Theorem \ref{1Dsolution}, let us assume that $u$ is a solution to  \eqref{equation1} satisfying \eqref{inftyassumptions}, such that
\begin{itemize}
\item[(i)] $u$ is monotone non-decreasing in $\Sigma_\lambda$

\

\text{and}

\

\item [(ii)] $u \leq u_\lambda$ in $\Sigma_\lambda$.
\end{itemize}
Then $u < u_\lambda$ in $\Sigma_\lambda \setminus \mathcal{Z}_{f(u)}$.
\end{prop}

\begin{proof}
By \eqref{inftyassumptions}, given $0<\delta_0<1$ there exists $M_0=M_0(\delta_0)>0$, with $\lambda>-M_0$, such that  $u(x)=u(x',x_N)<-1+\delta_0$ in $\{x_N<- M_0\}$ and $u_\lambda(x)=u(x',2\lambda -x_N)>1-\delta_0$ in $\{x_N<- M_0\}$. We fix $\delta_0$ sufficiently small such that $f'(u)<-L$ in $\{x_N<- M_0\}$, for some $L>0$. Arguing by contradiction, let us assume that there exists $P_0
=(x_0',x_{N,0}) \in \Sigma_\lambda \setminus \mathcal{Z}_{f(u)}$ such
that $u(P_0)=u_\lambda(P_0)$. Let $\mathcal{U}_0$ the connected component of $\Sigma_\lambda \setminus \mathcal{Z}_{f(u)}$ containing $P_0$. By Theorem \ref{SCPLucioeDino}, since $u(P_0)=u_\lambda(P_0)$, we deduce that $\mathcal U_0$ is a local symmetry region, i.e. $u\equiv u_\lambda$ in $\mathcal{U}_0$.


We notice that, by construction, $u < u_\lambda$ in
$\Sigma_{-M_0}$, since $u(x)<-1+\delta_0$ and
$u_\lambda(x)=u(x',2\lambda - x_N) > 1-\delta_0$ in $\Sigma_{-M_0}$. Since $\mathcal{U}_0$ is an open set of $\Sigma_\lambda \setminus \mathcal{Z}_{f(u)}$ (and also of $\R^N$) there exists $\rho_0=\rho_0(P_0)>0$ such
that
\begin{equation}\label{eq:pallaprimopunto}
B_{\rho_0}(P_0) \subset \mathcal{U}_0.
\end{equation}

\begin{figure}[htbp]
	\centering
	\includegraphics[scale=.4]{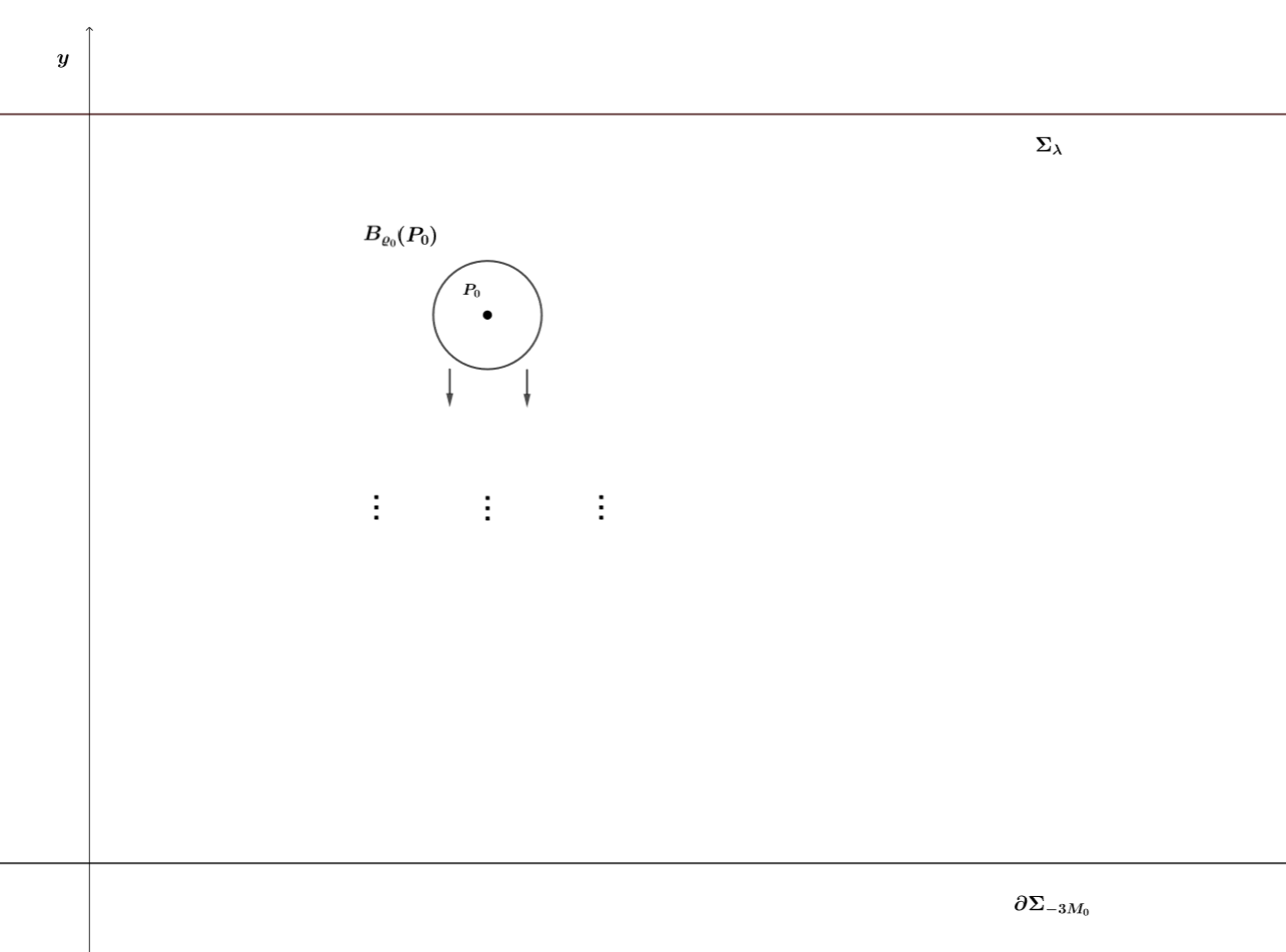}\\
	\caption{The slided ball $B_{\rho_0}(P_0)$ \label{fg:slided}}
\end{figure}

We can slide $B_{\rho_0}$ in $\mathcal U_0$, towards to $-\infty$ in the $y$-direction and keeping its centre on the line $\{x'=x'_0\}$ (see Figure \ref{fg:slided}), until it touches  for the first time  $\partial \mathcal{U}_0$ at some point $z_0\in  \mathcal{Z}_{f(u)}$. In Figure \ref{fg:1contact}, we show some possible examples of {\em first contact point} with the set $\mathcal{Z}_{f(u)}$.
\begin{figure}[htbp]
	\centering
	\includegraphics[scale=.35]{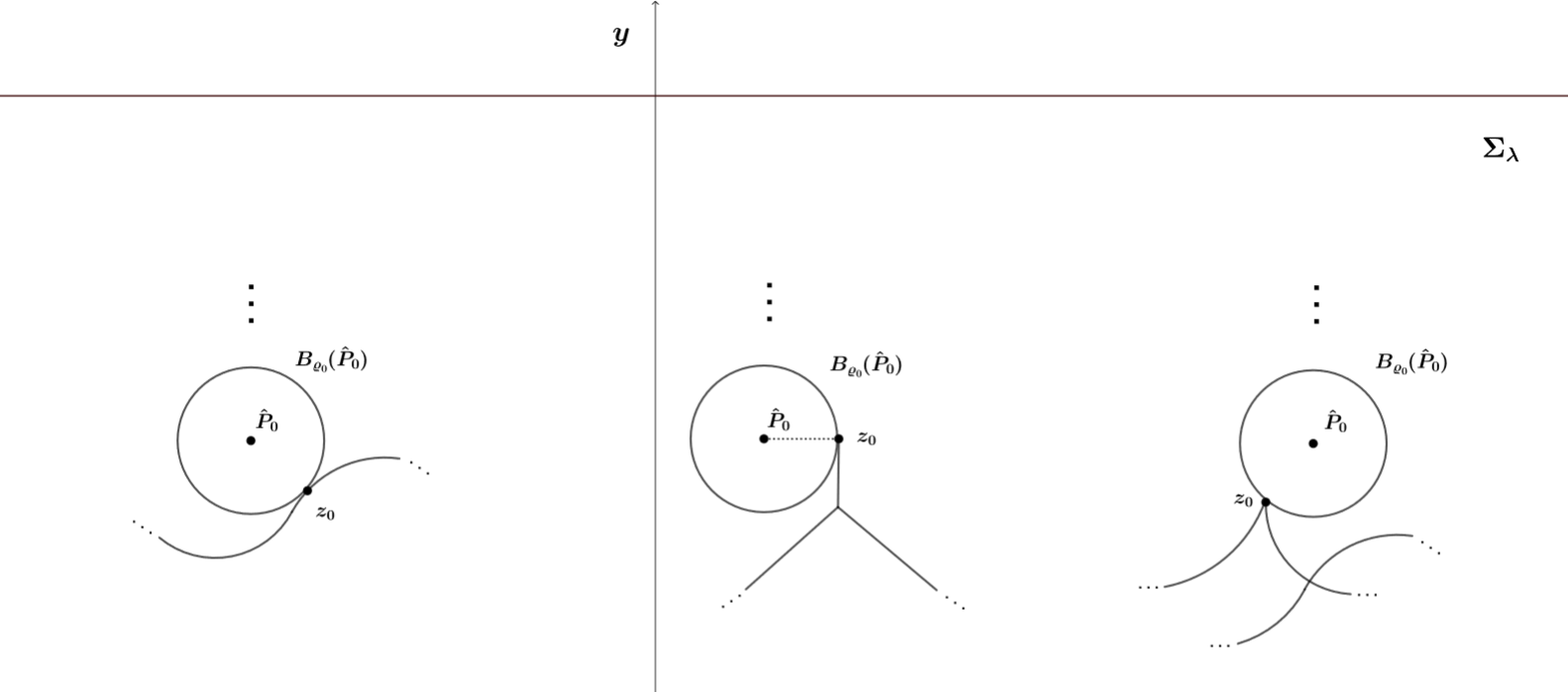}\\
	\caption{The first contact point $z_0$ \label{fg:1contact}}
\end{figure}

Now we consider the function
$$w_0(x):=u(x)-u(z_0)$$
and we observe that $w_0(x) \neq 0$ for every $x \in B_{\rho_0}(\hat P_0)$, where $\hat P_0$ is the new centre of the slided ball. In fact,  if
this is not the case there would exist a point $\bar z \in  B_{\rho_0}(\hat P_0)$
such that $w_0(\bar z)=0$, but this is in contradiction with the
fact that $\mathcal{U}_0 \cap \mathcal{Z}_{f(u)}= \emptyset$.
We
have to distinguish two
cases. Since $p < 2$ and $f$ is locally Lipschitz, we have that

\

{{Case 1:} If $w_0(x) > 0$ in $B_{\rho_0}(\hat P_0)$, then
\begin{equation}\nonumber
\begin{cases}
\Delta_p w_0 \leq C w_0^{p-1} & \quad \text{in} \ B_{\rho_0}(\hat P_0)\\ 
w_0>0 & \quad \text{in} \ B_{\rho_0}(\hat P_0)\\
w(z_0)=0 & \quad z_0 \in \partial B_{\rho_0}(\hat P_0),
\end{cases}
\end{equation}
where $C$ is a positive constant.}

{{Case 2:} If $w_0(x) < 0$ in $B_{\rho_0}(\hat P_0)$, setting $v_0 = - w_0$ we have
\begin{equation}\nonumber
\begin{cases}
\Delta_p v_0 \leq C v_0^{p-1} & \quad \text{in} \ B_{\rho_0}(\hat P_0)\\ 
v_0>0 & \quad \text{in} \ B_{\rho_0}(\hat P_0)
\\
v_0(z_0)=0 & \quad z_0 \in \partial B_{\rho_0}(\hat P_0),
\end{cases}
\end{equation}
where $C$ is a positive constant.}

In both cases, by the H\"opf boundary lemma (see e.g. \cite{pucser, Vaz}), it
follows that $|\nabla w(z_0)|=|\nabla u (z_0)| \neq 0$.

\

 Using the Implicit Function Theorem we deduce that the set $\{u=u(z_0)\}$
is a  smooth manifold near $z_0$. Now we want to prove that
\[u_{x_N}(z_0)
> 0\] and actually that the set $\{u=u(z_0)\}$ is a graph in the $y$-direction near the point $z_0$. By our assumption we know that $u_{x_N}(z_0):=u_y(z_0) \geq 0$.
According to \cite{DSCalcVar, DSJDE} and \eqref{eq:linearizzatoSCP}, the linearized operator of \eqref{equation1} is
well defined
\begin{equation}\label{linearized}
\begin{split}
L_u(u_y, \varphi) &\equiv \\
\int_{\Sigma_\lambda} [|\nabla u|^{p-2} (\nabla u_y, \nabla \varphi) +
(p-2) &|\nabla u|^{p-4} (\nabla u, \nabla u_y)(\nabla u, \nabla
\varphi)] \, dx +\\
&- \int_{\Sigma_\lambda} f'(u) u_y \varphi \, dx
\end{split}
\end{equation}
for every $\varphi \in C^1_c(\Sigma_\lambda)$. Moreover $u_y$  satisfies the
linearized equation \eqref{linearizedequation}, i.e.
\begin{equation}\label{linearizedEq}
L_u(u_y, \varphi) =0 \quad \forall \varphi \in C^1_c(\Sigma_\lambda).
\end{equation}
Let us set $z_0=(z_0',y_0)$.  We
have two possibilities: $u_y(z_0)=0$ or $u_y(z_0)>0$.

\

{\textit{Claim:} We show that the case $u_y(z_0)=0$ is not possible. }

{If $u_y(z_0)=0$, then $u_y(x) \equiv 0$ in all $B_{\hat \rho}(z_0)$ for some positive ${\hat \rho}$; to prove this we use the fact that $|\nabla u(z_0)|\neq 0$, $u\in C^{1,\alpha}$ and that Theorem \ref{ClassicSMPlinearized} holds.}

{By construction of $z_0$ there exists $\varepsilon_1 >0$ such that every point $z \in \mathcal{S}_1 := \{(z'_0,t) \in \mathcal{U}_0 : y_0 < t < y_0+ \varepsilon_1\}$ has the following properties:
\begin{enumerate}
	\item $z \in \overline{\mathcal{U}_0}$, since the ball is sliding along the segment $\mathcal{S}_1$;
	
	\item $z \not \in \partial \mathcal{U}_0$, since $z_0$ is the first contact point with $\partial \mathcal{U}_0$.
\end{enumerate}
In particular, for every $z \in \mathcal{S}_1$ we have
\begin{equation}\label{sliding}
z \in \overline{\mathcal{U}_0} \setminus \partial \mathcal{U}_0 = \mathcal{U}_0.
\end{equation}
Since $|\nabla u(z_0)|\neq 0$ and $u\in C^{1,\alpha}$, by Theorem \ref{ClassicSMPlinearized} it follows that there exists $0< \varepsilon_2 < \varepsilon_1$ such that
$$u_y(x)=0 \quad \forall x \in B_{\varepsilon_2}(z_0).$$
Let us consider $\mathcal{S}_2:=\{(z'_0,t) \in \mathcal{U}_0 : y_0 < t < y_0+ \varepsilon_2\}$; by definition $\mathcal{S}_2 \subset \mathcal{S}_1$ and every point of $\mathcal{S}_2$ belongs also to $\mathcal{Z}_{f(u)}$, since $u(z)=u(z_0)$ for every $z \in \mathcal{S}_2$ and $z_0 \in \mathcal{Z}_{f(u)}$ by our assumptions. But this gives a contradiction with \eqref{sliding}.

}
%

\

From what we have seen above,  we have $|\nabla u (z_0)| \neq 0$ and
hence there exists a ball $B_r(z_0)$ where $|\nabla u (x)| \neq 0$
for every $x \in B_r(z_0)$. By Theorem \ref{classicalSCP} it follows that $u \equiv u_\lambda$ in
$B_r(z_0)$ namely $u \equiv u_\lambda$ in a neighborhood of the point $z_0\in \partial \mathcal U_0$. Since $u_y (z_0)>0$ and $\mathcal{N}_f$ is finite
\[B_r(z_0)\cap \Big((\Sigma_\lambda \setminus \mathcal{Z}_{f(u)})\setminus \mathcal{U}_0\Big)\neq \emptyset\]
and $u_y(x) >0$ in $B_r(z_0)$, as consequence,   the set $\{u=u(z_0)\}$ is a graph in the $y$-direction in a neighborhood  of  the point $z_0$.
Now we have to distinguish two cases:
\begin{itemize}
\item[Case 1:]  $u(z_0)=\min \Big [\mathcal{N}_f\setminus\{-1\}\Big].$

\noindent Define the sets \[ \mathcal{C}_1:=\Big\{x\in \mathbb R^N \,:\,  x'\in \left( B_r(z_0) \cap \{y:=y_0\} \right) \quad \text{and}\quad u(x)<u(z_0)\Big \} \]
\[ \mathcal{C}_2:= B_r(z_0) \cup \big( (B_r(z_0) \cap \{y:=y_0\}) \times (-\infty, y_0)\big)\]
and
\[\mathcal{C}=\mathcal{C}_1\cap \mathcal{C}_2.\]
We observe that $\mathcal{C}$ is an open unbounded path-connected set (actually a deformed cylinder), see Figure \ref{fg:3lastcomponent}. Since $f(u(z_0))$ has the right sign, by Theorem \ref{SCPLucioeDino} it follows that $u\equiv u_\lambda$ in $\mathcal C$ and this in  contradiction with the uniform limit conditions \eqref{inftyassumptions}.
\begin{figure}[htbp]
	\centering
	\includegraphics[scale=.52]{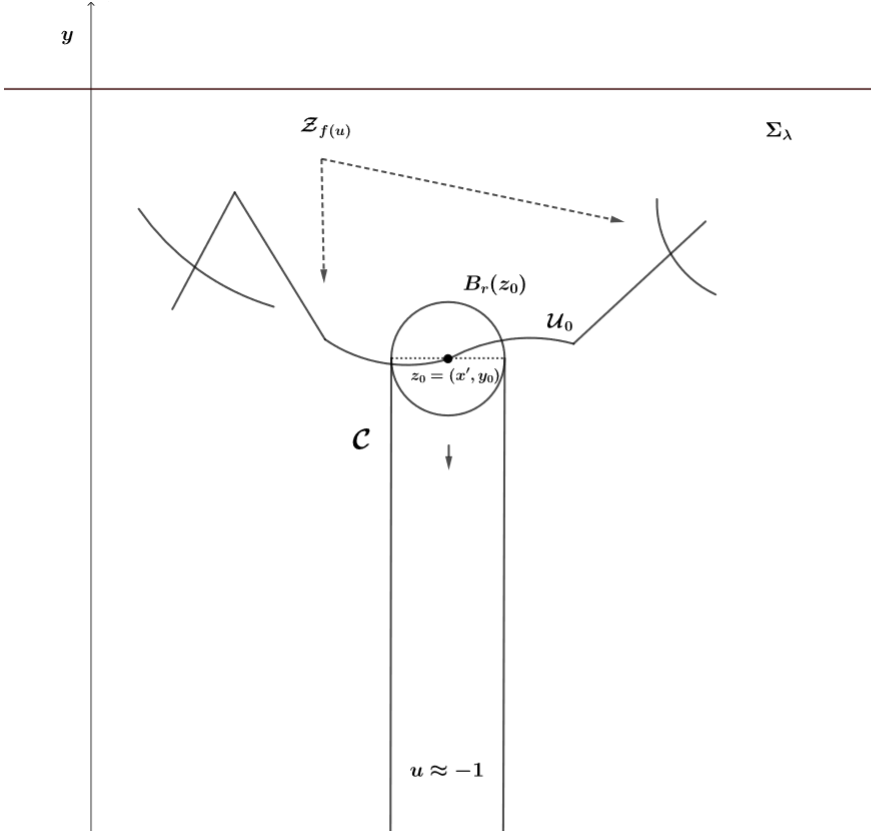}\\
	\caption{Case 1:  $u(z_0)=\min \Big [\mathcal{N}_f\setminus\{-1\}\Big]$ \label{fg:3lastcomponent}}
\end{figure}

\

\item [Case 2:] $u(z_0)>\min \Big [\mathcal{N}_f\setminus\{-1\}\Big]$.

\noindent In this case the open ball $B_r(z_0)$ must intersect another connected component  (i.e. $\not\equiv \mathcal{U}_0 $) of $\Sigma_\lambda \setminus \mathcal{Z}_{f(u)}$, such that  $u \equiv u_\lambda$ in  a such component, see Figure \ref{fg:3morecomponents}. Here we used the fact that near the (new) first contact point, the corresponding level set is a graph in the $y$-direction. Now, it is clear that repeating a finite number of times the argument leading to the existence of the touching point $z_0$, we can find a touching point $z_m$ such that \[u(z_m)=\min \Big [\mathcal{N}_f\setminus\{-1\}\Big].\] The contradiction then follows exactly as in Case 1.
\begin{figure}[htbp]
	\centering
	\includegraphics[scale=.54]{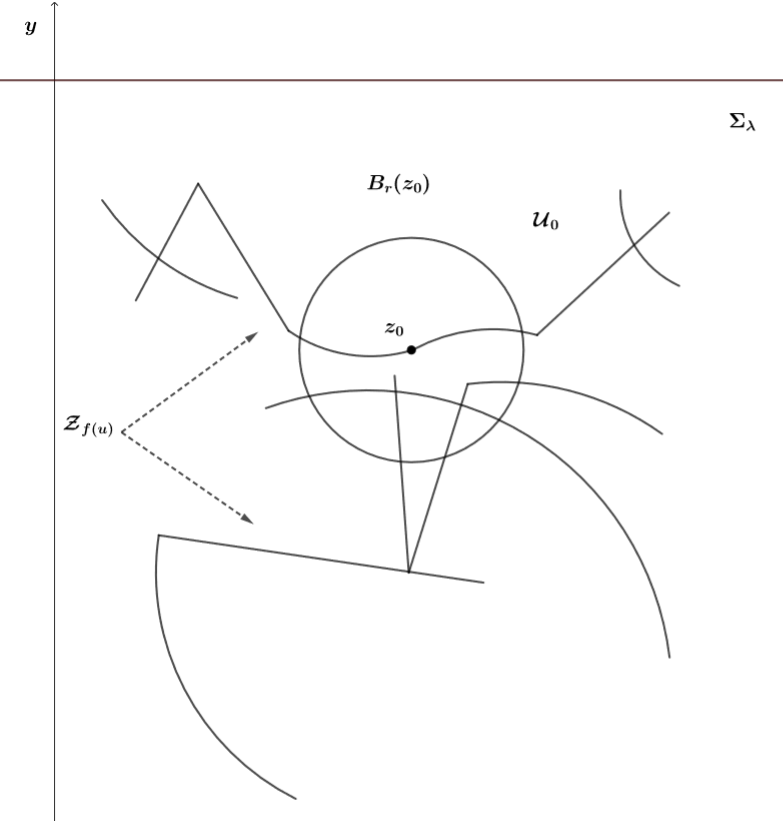}\\
	\caption{Case 2:  $u(z_0)>\min \Big [\mathcal{N}_f\setminus\{-1\}\Big]$ \label{fg:3morecomponents}}
\end{figure}
\end{itemize}
 Hence $u < u_\lambda$ in $\Sigma_\lambda \setminus
\mathcal{Z}_{f(u)}$.
\end{proof}

To prove Proposition \ref{monotonicitybis} we need of the following  result:
\begin{lem}\label{gradBelow}
	Under the assumption of Theorem \ref{1Dsolution}, let $u$ be a solution to \eqref{equation1}. Then there exist $M_0=M_0(p,f, N, \|\nabla u\|_{L^{\infty}(\mathbb{R^N})})>0$ sufficiently large such that for every $M\geq M_0$ there exits  a constant $C^*=C^*(M)>0$ such that
	\begin{equation}
	{|\nabla u|\geq\partial_{x_N} u \geq C^*> 0 \quad \text{in} \; \{-M-1<x_N < -M+1\}.}
	\end{equation}
\end{lem}

\begin{proof}
Performing the moving plane procedure, using \eqref{inftyassumptions} and {\bf ($h_f$)}, by the  Proposition  \ref{weakcomparisonprinciple} with $v=u_\lambda$ and $\Sigma=\Sigma_\lambda$,  we infer that there exists a constant $M_0=M_0(p,f, N, \|\nabla u\|_{L^{\infty}(\mathbb{R^N})})>0$   such  that $\partial_{x_N} u \geq 0$ in $\{x_N < -M_0+1\}$. Now we can assume $$\mathcal{Z}_{f(u)} \cap \{x_N < -M_0+1\} = \emptyset,$$ then by Theorem \ref{SMPlinearized} it follows that $\partial_{x_N} u > 0$ in $\{x_N < -M_0+1\}$, since the case $\partial_{x_N} u = 0$ would imply a contradiction, i.e. $u(x)=-1$  in $\{x_N < -M_0+1\}$ . We observe that in particular it holds $|\nabla u| \geq \partial_{x_N} u>0$ in $\{-M_0-1<x_N<-M_0+1\}$. We want to prove that for all $M\geq M_0$,  there exists  $C^*=C^*(M)>0$ such that $\partial_{x_N} u \geq C^*>0$ in $\{-M-1<x_N<-M+1\}$.

Arguing by contradiction let us assume that there exists a sequence of point $P_n=(x'_n,x_{N,n})$, with $-M-1<x_{N,n}<-M+1$ for every $n \in \N$, such that $\partial_{x_N} u (P_n) \rightarrow 0$ as $n \rightarrow + \infty$ in $\{-M-1<x_N<-M+1\}$. Up to subsequences, let us assume that
$$x_{N,n} \rightarrow \bar x_N \; \text{with} \; -M-1 \leq \bar x_N \leq -M+1.$$
Let us now define
$$\tilde u_n(x',x_N):=u(x'+x'_n, x_N)$$
so that $\|\tilde u_n\|_\infty = \| u \|_\infty \leq 1$. By standard
regularity theory, see \cite{DB, T}, we have that
$$\|\tilde u_n\|_{C^{1,\alpha}_{loc}(\R^N)} \leq C$$
for some $0 < \alpha < 1$. By Ascoli's Theorem we have
$$\tilde u_n \overset{C^{1,\alpha'}_{loc}(\R^N)}{\longrightarrow} \tilde u$$
up to subsequences, for $\alpha' < \alpha$. By construction $\partial_{x_N} \tilde u \geq 0$ and $\partial_{x_N} \tilde u (0,\bar x_N) = 0$, hence by Theorem \ref{ClassicSMPlinearized} it follows that $\partial_{x_N}\tilde u = 0$ in $\{-M-1<x_N<-M+1\}$ and therefore $\partial_{x_N}\tilde u = 0$ in all $\{(x',x_n)\,:\,x_N<-M+1\}$ by Theorem \ref{SMPlinearized}, since $\mathcal{Z}_{f(u)} \cap \{x_N < -M_0+1\} = \emptyset$. This gives a contradiction (by Theorem \ref{ClassicSMPlinearized})  with the fact that $\displaystyle\lim_{x_N\rightarrow -\infty}u(x',x_N)=-1$ (this implies that $\displaystyle\lim_{x_N\rightarrow -\infty} \tilde u(x',x_N)=-1$ ), see Remark \ref{rem:1}.
 \end{proof}
With the notation introduced above, we set
\begin{equation}\label{eq:miamnamgnam}
\Lambda := \{\lambda \in \R \ | \ u \leq u_t \; \text{in} \ \Sigma_t \; \forall t < \lambda\}.
\end{equation}

Note that, by Proposition \ref{weakcomparisonprinciple} (with $v=u_t)$, it follows
that $\Lambda \neq \emptyset$, hence we can define
\begin{equation}\label{eq:suplambda}
\bar \lambda := \sup \Lambda.\end{equation}
Moreover it is important to say that by the continuity of $u$ and $u_\lambda$, it follows that
$$u \leq u_{\bar \lambda} \quad \text{in} \; \Sigma_{\bar \lambda}.$$
The proof of the fact that  $u(x',x_N)$ is monotone increasing in the $x_N$-direction in the entire space $\mathbb R^N$ is done once show that $\bar
\lambda = + \infty$. To do this we assume by contradiction that $\bar \lambda < + \infty$, and we prove a crucial result,
which allows us to localize the support of $(u-u_{\bar \lambda})^+$.
This localization, that we are going to obtain, will be useful to
apply the weak comparison principle given by Proposition \ref{weakcomparisonprinciple} and  Theorem \ref{compprinciplenarrow}.

\begin{prop}\label{localizationSupport}
Under the assumption of Theorem \ref{1Dsolution}, let $u$ be a solution to \eqref{equation1}. Assume that $\bar
\lambda < + \infty$ (see \eqref{eq:suplambda}) and set
\begin{equation}\nonumber
W_\varepsilon:=(u-u_{\bar \lambda + \varepsilon}) \chi_{\{x_N \leq
\bar \lambda + \varepsilon\}}.
\end{equation}\label{WepsilonBis}
%
%
%
Let $M,\kappa>0$ be  such that  $M>2|\bar \lambda|$. Then  for all   $\mu\in (0, (\bar\lambda +M)/2)$ there exists $\bar \varepsilon > 0$ such that for every $0 < \varepsilon < \bar \varepsilon$
\begin{equation}\label{supportWepsilon}
\text{\emph{supp}}\, W^+_\varepsilon \subset \{x_N \leq -M\} \cup
\{ \bar \lambda - \mu \leq x_N \leq \bar \lambda + \varepsilon \} \cup
\{ |\nabla u| \leq \kappa \}.
\end{equation}
\end{prop}

\begin{proof}
Assume by contradiction that \eqref{supportWepsilon} is false, so that there
exists $\mu >0$  in such a way that, given any $\bar\varepsilon>0$, we find
$0<\varepsilon \leq \bar\varepsilon$ so that there exists a
corresponding $x_\varepsilon = (x'_\varepsilon, x_{N, \varepsilon})$
such that
$$u(x'_\varepsilon, x_{N, \varepsilon}) \geq u_{\bar
\lambda + \varepsilon} (x'_\varepsilon, x_{N, \varepsilon}),$$
with $x_\varepsilon = (x'_\varepsilon, x_{N, \varepsilon})$
 belonging to the set
\[\{(x',x_N)\in \mathbb R^N\,:\,M < x_{N, \varepsilon} < \bar \lambda - \mu \}\]
and such that $ |\nabla u(x_\varepsilon)| \geq \kappa$.

Taking $\bar\varepsilon = {1}/{n}$, then there exists
$\varepsilon_n \leq 1/n$ going to zero, and a
corresponding  sequence $$x_n = (x'_n, x_{N,n}) = (x'_{\varepsilon_n}, x_{N,
\varepsilon_n})$$ such that
$$u(x'_n, x_{N, n}) \geq
u_{\bar \lambda + \varepsilon_n} (x'_n, x_{N, n})$$
with  $-M < x_{N,n} < \bar \lambda - \mu$. Up to subsequences, let us
assume that
$$x_{N,n} \rightarrow \bar x_N \; \text{with} \;
-M \leq \bar x_N \leq \bar \lambda - \mu.$$

Let us define
$$\tilde u_n(x',x_N):=u(x'+x'_n, x_N)$$
so that $\|\tilde u_n\|_\infty = \| u \|_\infty \leq 1$. By standard
regularity theory, see \cite{DB, T}, we have that
$$\|\tilde u_n\|_{C^{1,\alpha}_{loc}(\R^N)} \leq C$$
for some $0 < \alpha < 1$. By Ascoli's Theorem we have
$$\tilde u_n \overset{C^{1,\alpha'}_{loc}(\R^N)}{\longrightarrow} \tilde u$$
up to subsequences, for $\alpha' < \alpha$. By construction it follows that
\begin{itemize}
\item $\tilde u \leq \tilde u_{\bar \lambda} \quad \text{in} \
\Sigma_{\bar \lambda}$;

\item $\tilde u(0,\bar x_N) = \tilde u_{\bar \lambda} (0, \bar x_N)$;

\item $| \nabla \tilde u(0, \bar x_N)| \geq \kappa$.
\end{itemize}

Since $|\nabla \tilde{u} (0, \bar x_N)| \geq \kappa$ there exists $\rho >0$ and a ball $B_\rho(0,\bar x_N) \subset \Sigma_{\bar \lambda}$ such that $|\nabla u(x)| \neq 0$ for every $x \in B_\rho(0,\bar x_N)$. Now, if $\tilde{u}(0,\bar x_N) \in \mathcal{Z}_{f(u)}$, since $\tilde u$ is non constant  in  $B_\rho(0,\bar x_N)$, there exists  $P_0 \in B_\rho(0,\bar x_N)$ such that $u(P_0) \not \in \mathcal{Z}_{f(u)}$.
By Theorem \ref{classicalSCP} it follows that
\begin{equation}\label{disuguaglianza}
\tilde{u} \equiv \tilde{u}_{\bar{\lambda}} \quad \text{in}\,\, B_\rho(0,\bar x_N).\end{equation}
On the other hand, by Proposition \ref{connectedness} it follows that
\begin{equation}\nonumber
\tilde{u} < \tilde{u}_{\bar \lambda} \qquad \text{in} \,\,\Sigma_{\bar{\lambda}} \setminus \mathcal{Z}_{f(u)}.
\end{equation}
This gives a contradiction with \eqref{disuguaglianza}. Hence we have \eqref{supportWepsilon}.
\end{proof}

\begin{proof}[Proof of Proposition \ref{monotonicitybis}]  {Let us assume by contradiction that $\bar\lambda<+\infty$, see  \eqref{eq:suplambda}. Let $\hat M>0$ be such that Proposition \ref{weakcomparisonprinciple} and Lemma \ref{gradBelow}   apply. Let $C^*=C^*(\hat M)$ be the constant given in Lemma  \ref{gradBelow}.
By  Proposition \ref{localizationSupport} (choose $M=4 \hat M+1$ there, redefining $\hat M$  if necessary) we have that
\begin{equation}\label{eq:primac'era}\text{\emph{supp}}\, W^+_\varepsilon \subset \{x_N \leq -4\hat M-1\} \cup
\{-4\hat M +1\leq x_N \leq \bar \lambda + \varepsilon \},
\end{equation}
where $W_\varepsilon:=(u-u_{\bar \lambda + \varepsilon}) \chi_{\{x_N \leq
\bar \lambda + \varepsilon\}}$. In particular, to get \eqref{eq:primac'era}, we choose $\kappa$ in Proposition~\ref{localizationSupport} such that $2\kappa=C^*$. Then we deduce that
\begin{equation}\label{eq:secretn01}
u\leq u_{\bar\lambda+\varepsilon}\,\,\text{in}\,\, \{(x,x_N)\in \mathbb R^N\,:\ -4\hat M -1<x_N<-4\hat M +1\}.\end{equation}
Using \eqref{eq:secretn01},  we can apply Proposition \ref{weakcomparisonprinciple} in $\{x_N < -4\hat M-1\}$  and therefore, together Lemma \ref{gradBelow} and Proposition \ref{localizationSupport}, we actually  deduce\[\text{\emph{supp}}\, W^+_\varepsilon \subset \{-4\hat M +1\leq x_N \leq \bar \lambda + \varepsilon \}.\]
In particular, if we look to  \eqref{supportWepsilon},  we deduce that $\text{\emph{supp}}\, W^+_\varepsilon$ must belong to the set \[A:=\big\{\{ \bar \lambda - \mu \leq x_N \leq \bar \lambda + \varepsilon \} \cup
\{ |\nabla u| \leq \kappa \}\big\}\cap \big\{x_N\geq -4\hat M +1\big\}. \]
We now apply Theorem \ref{compprinciplenarrow} in the set $A$. Let us choose (in  Theorem \ref{compprinciplenarrow})$$L_0=1+\|\nabla u\|_{L^{\infty}(\mathbb{R}^N)}$$ and take
$\tau_0=\tau_0(p,\bar\lambda, \hat M, N, L_0)>0$
and
$\epsilon_0=\epsilon_0(p,\bar\lambda, \hat M, N, L_0)>0$ as in Theorem \ref{compprinciplenarrow}.
Let $\mu,
\varepsilon$ in Proposition \ref{localizationSupport} such that $2(\mu +\varepsilon)< \tau_0$ and let us redefine $\kappa$ eventually such that $\kappa:=\min\{C^*/2, \epsilon_0\}$.  We finally apply Theorem \ref{compprinciplenarrow} concluding  that actually $W_\varepsilon^+=0$ in the set $A$. This gives a contradiction, in view of the definition \eqref{eq:suplambda} of $\bar \lambda$. Consequently we deduce that $\bar \lambda=+\infty$. This implies the monotonicity of $u$, that is $\partial_{x_N} u \geq 0$ in $\R^N$. By Theorem \ref{SMPlinearized}, it follows that
$$\partial_{x_N} u > 0 \quad \text{in} \, \R^N \setminus
\mathcal{Z}_{f(u)},$$
since by Lemma \ref{lem:utile},  the case $\partial_{x_N} u\equiv 0$ in some connected component, say $\mathcal U$, of $ \R^N \setminus
\mathcal{Z}_{f(u)}$ can not hold.}
\end{proof}

\section{1-Dimensional Symmetry}\label{sec: 1sim}
{In this section we pass from the monotonicity in $x_N$ to the monotonicity in all the directions of the upper hemisphere $\S^{N-1}_+$ defined in \eqref{eq:hemisphere}. We refer to  \cite{Far99} for the case of the Laplacian operator, where in the proof  the linearity of the operator was crucial. Here we have to take into account the  singular nature and the nonlinearity of the operator  $p$-Laplacian.
\begin{lem}\label{monCone}
Under the same assumption of Theorem \ref{1Dsolution}, given $\rho>0$ and $k>0$, we define
\[\Sigma_k^\rho:=\{x \in \R^N \ |  -k < x_N < k\} \cap \{|\nabla u| > \rho\}.\]
Assume $\eta \in \S^{N-1}_+ $ and suppose that
\begin{equation}\label{eq:muratoricolonne}
\partial_\eta u \geq 0\quad \text{in}\,\, \mathbb R^N \quad \text{and} \quad \partial_\eta u>0\quad \text{in}\,\, \mathbb R^N\setminus \mathcal{Z}_{f(u)}.
\end{equation}
Then, there exists an open neighbourhood $\mathcal O_\eta$ of $\eta$ in $\S^{N-1}_+$,
 such that
\begin{equation}\label{dirDer1}
\partial_\nu u = (\nabla u, \nu) > 0 \quad \text{in} \, \, \Sigma_k^\rho,
\end{equation}
for every $\nu\in \mathcal O_\eta$.
\end{lem}}
\begin{proof}
{Arguing by contradiction let us assume that there exist two sequences $\{P_m\}~\in \mathbb R^N$ and $\{\nu_m\} \in \S^{N-1}_+$ such that, for every $m \in \N$ we have that $P_m=(x'_m,x_{N,m}) \in \Sigma_k^\rho$,  $|(\nu_m, \eta) -1 | < 1/m$
and $\partial_{\nu_m}u(P_m) \leq 0$. Since $-k < x_{N,m} < k$ for every $m \in \N$, then up to subsequences $x_{N,m} \rightarrow \bar x_N$. Now, let us define
$$\tilde u_m(x', x_N):=u(x'+x'_m,  x_N)$$
so that $\|\tilde u_m\|_\infty = \| u \|_\infty \leq 1$. By standard
regularity theory, see \cite{DB, T}, we have that
$$\|\tilde u_m\|_{C^{1,\alpha}_{loc}(\R^N)} \leq C.$$
By Ascoli's Theorem, via a standard diagonal process, we have, up to subsequences
$$\tilde u_m \overset{C^{1,\alpha'}_{loc}(\R^N)}{\longrightarrow} \tilde u,$$
for some $0 <\alpha'< \alpha$.}

{By uniform convergence and \eqref{eq:muratoricolonne}  it follows that
\[\partial_{\eta} \tilde u(0,\bar x_N) = 0\quad \text{and}\quad | \nabla \tilde u(0, \bar x_N)| \geq \rho.\]
\begin{itemize}
\item If $P_0:=(0,\bar x_N) \in  \mathcal{Z}_{f(\tilde u)}$, since {$|\nabla \tilde u (0, \bar x_N)| \geq \rho$}, then there exists a ball $B_r(P_0)$ such that $|\nabla \tilde u (x)| \neq 0$ for every $x \in B_r(P_0)$. By Theorem \ref{ClassicSMPlinearized}, applied having in mind that $|\nabla \tilde u (x)| \neq 0$ in $B_r(P_0)$,   it follows that $\partial_{\eta} \tilde u (x) = 0$ for every $x \in B_r(P_0)$. In particular  $\partial_{\eta} \tilde u (x) = 0$ for every $x \in B_r(P_0) \cap \left(  \Sigma_k^\rho \setminus \mathcal{Z}_{f(\tilde u)} \right)$, hence by Theorem \ref{SMPlinearized} we deduce that $\partial_{\eta} \tilde u \equiv 0$ in  the connected component $\mathcal U$ of $\Sigma_k^\rho \setminus \mathcal{Z}_{f(\tilde u)}$ containing $B_r(P_0)$ (possibly redefining $r$),  but this is in contradiction with Lemma \ref{lem:utile}.
	
\item If $P_0 \in  \Sigma_k^\rho \setminus \mathcal{Z}_{f(\tilde u)}$ by Theorem \ref{SMPlinearized} it follows that $\partial_{\eta} \tilde u > 0$ in the connected component of  $\R^N \setminus \mathcal{Z}_{f(\tilde u)}$ containing the point $P_0$. Indeed  the case $\partial_{\eta} \tilde u \equiv 0$ in  the connected component of $\R^N \setminus \mathcal{Z}_{f(\tilde u)}$ containing $P_0$ can not hold since Lemma \ref{lem:utile}.
\end{itemize}
Hence we deduce \eqref{dirDer1}.}
\end{proof}

Having in mind the previous lemma, now we are able to prove the monotonicity in a small cone of direction around $\eta$ in the entire space.

\begin{prop}\label{monRncone}
{Under the assumption of Theorem \ref{1Dsolution}, assume $\eta \in \S^{N-1}_+ $ such that  ${\partial_\eta u}>0$ in $\mathbb R^N\setminus \mathcal{Z}_{f(u)}$.  Then, there exists an open neighbourhood $\mathcal O_\eta$ of $\eta$ in $\S^{N-1}_+$,
 such that
\begin{equation}\label{eq:ultimasergioaereo}
\partial_\nu u = u_\nu \geq 0 \quad \text{in} \,\, \R^N \quad \text{and} \quad \partial_\nu u = u_\nu > 0\quad \text{in} \,\, \R^N \setminus \mathcal{Z}_{f(u)},\end{equation}
for every $\nu\in \mathcal O_\eta$.}
\end{prop}

\begin{proof}
We fix $\tilde \delta>0$  and let  $k=k(\tilde \delta)>0$ be such that $u < -1 + \tilde \delta$ in $\{x_N< -k\}$,  $u > 1- \tilde \delta$ in $\{x_N>k\}$ and \eqref{eq:derivatafnegativa} holds  in $\{|x_N| > k\}$. By Lemma \ref{monCone} it follows that for all $\rho>0$ one has
$$\text{supp}\, (u_\nu^-) \subseteq \Big( \{|x_N|\geq k\} \cup \left(\{-k<x_N<k\}\cap \{|\nabla u|\leq \rho\}\right)\Big).$$
For simplicity of exposition we set
\[A:=\{|x_N|\geq k\}\quad  \text{and} \quad D:=\big(\{-k<x_N<k\}\cap \{|\nabla u|\leq \rho\}\big).\] Our claim is to show that $u_\nu^-=0$ in $A \cup D$. In order to do this we split the proof in two part.

{\em Step 1}. We show that $u_\nu^-=0$ in $A$.\\
We set
\begin{equation}\label{testA}
\varphi:=(u_\nu^-)^\alpha \varphi_R^{2} \chi_{\mathcal{A}(2R)}
\end{equation}
where $\alpha >1$,  $R >0$ large,  $\mathcal{A}(2R):=A \cap B_{2R}$ and $\varphi_R$ is a standard cutoff function such that
$0\leq \varphi_R \leq 1$ on $\R^N$, $\varphi_R = 1$ in $B_R$,
$\varphi_R = 0$ outside $B_{2R}$, with $|\nabla \varphi_R| \leq 2/R$
in $B_{2R} \setminus B_R$. First of all we notice that $\varphi$
belongs to $W^{1,p}_0(\mathcal{A}(2R))$. To see this, use the definition of $\varphi_R$  and note that by Lemma \ref{gradBelow} and Lemma \ref{monCone}, it follows  that $u_\nu^-=0$ on the hyperplanes $|x_N|=k$, namely on $\partial A$.

According to \cite{DSCalcVar, DSJDE}, the linearized operator is
well defined

\begin{equation}\label{linearizedA}
\begin{split}
L_u(u_\nu, \varphi) &\equiv \\
\int_{\R^N} [|\nabla u|^{p-2} (\nabla u_\nu, \nabla \varphi) +
(p-2) &|\nabla u|^{p-4} (\nabla u, \nabla u_\nu)(\nabla u, \nabla
\varphi)] \, dx +\\
&- \int_{\R^N} f'(u) u_\nu \varphi \, dx
\end{split}
\end{equation}
for every $\varphi \in C^1_c(\R^N)$. Moreover it satisfies the
following equation
\begin{equation}\label{linearizedEqA}
L_u(u_\nu, \varphi) =0 \quad \forall \varphi \in C^1_c(\R^N).
\end{equation}
Taking $\varphi$ defined in \eqref{testA} in the previous equation, we obtain
\begin{equation}\label{computationdir}
\begin{split}
&\alpha \int_{\mathcal{A}(2R)} |\nabla u|^{p-2} (\nabla u_\nu, \nabla u_\nu^-) (u_\nu^-)^{\alpha-1} \varphi_R^2\\
& + 2 \int_{\mathcal{A}(2R)} |\nabla u|^{p-2} (\nabla u_\nu, \nabla \varphi_R) (u_\nu^-)^\alpha \varphi_R \\
&+\alpha (p-2)  \int_{\mathcal{A}(2R)} |\nabla u|^{p-4} (\nabla u, \nabla u_\nu)(\nabla u, \nabla u_\nu^-) (u_\nu^-)^{\alpha-1} \varphi_R^2 \, dx \\
&+ 2 (p-2)  \int_{\mathcal{A}(2R)} |\nabla u|^{p-4} (\nabla u, \nabla u_\nu)(\nabla u, \nabla \varphi_R) (u_\nu^-)^\alpha \varphi_R \, dx \\
&= \int_{\mathcal{A}(2R)} f'(u) (u_\nu^-)^{\alpha+1} \varphi_R^2 \, dx
\end{split}
\end{equation}
Making some computations we obtain
\begin{equation}\label{computationdir2}
\begin{split}
&\alpha \int_{\mathcal{A}(2R)} |\nabla u|^{p-2} |\nabla u_\nu^-|^2 (u_\nu^-)^{\alpha-1} \varphi_R^2 \, dx\\
& = -2 \int_{\mathcal{A}(2R)} |\nabla u|^{p-2} (\nabla u_\nu^-, \nabla \varphi_R) (u_\nu^-)^\alpha \varphi_R \, dx\\
&+\alpha (2-p)  \int_{\mathcal{A}(2R)} |\nabla u|^{p-4} (\nabla u, \nabla u_\nu^-)^2 (u_\nu^-)^{\alpha-1} \varphi_R^2 \, dx \\
&+ 2 (2-p)  \int_{\mathcal{A}(2R)} |\nabla u|^{p-4} (\nabla u, \nabla u_\nu^-)(\nabla u, \nabla \varphi_R) (u_\nu^-)^\alpha \varphi_R \, dx \\
&+  \int_{\mathcal{A}(2R)} f'(u) (u_\nu^-)^{\alpha+1} \varphi_R^2 \, dx \\
&\leq \alpha (2-p)  \int_{\mathcal{A}(2R)} |\nabla u|^{p-2} |\nabla u_\nu^-|^2 (u_\nu^-)^{\alpha-1} \varphi_R^2 \, dx \\
&+2(3-p) \int_{\mathcal{A}(2R)} |\nabla u|^{p-2} |\nabla u_\nu^-| \ |\nabla \varphi_R| (u_\nu^-)^\alpha \varphi_R \, dx\\
&+  \int_{\mathcal{A}(2R)} f'(u) (u_\nu^-)^{\alpha+1} \varphi_R^2\, dx.
\end{split}
\end{equation}
Now it is possible to rewrite \eqref{computationdir2} as follows
\begin{equation}\label{computationdir3}
\begin{split}
&\alpha (p-1)\int_{\mathcal{A}(2R)} |\nabla u|^{p-2} |\nabla u_\nu^-|^2 (u_\nu^-)^{\alpha-1} \varphi_R^2 \, dx\\
& \leq 2(3-p) \int_{\mathcal{A}(2R)} |\nabla u|^{p-2} |\nabla u_\nu^-| \ |\nabla \varphi_R| (u_\nu^-)^\alpha \varphi_R \, dx\\
&+  \int_{\mathcal{A}(2R)} f'(u) (u_\nu^-)^{\alpha+1} \varphi_R^2 \, dx.
\end{split}
\end{equation}
Exploiting the weighted Young inequality we obtain
\begin{equation}\label{computationdir4}
\begin{split}
&\alpha (p-1)\int_{\mathcal{A}(2R)} |\nabla u|^{p-2} |\nabla u_\nu^-|^2 (u_\nu^-)^{\alpha-1} \varphi_R^2 \, dx\\
& \leq 2(3-p) \int_{\mathcal{A}(2R)} |\nabla u|^\frac{p-2}{2} |\nabla u_\nu^-| \  (u_\nu^-)^\frac{\alpha-1}{2} |\nabla u|^\frac{p-2}{2} |\nabla \varphi_R| (u_\nu^-)^\frac{\alpha+1}{2} \varphi_R \, dx\\
&+  \int_{\mathcal{A}(2R)} f'(u) (u_\nu^-)^{\alpha+1} \varphi_R^2 \, dx\\
& \leq \sigma (3-p) \int_{\mathcal{A}(2R)} |\nabla u|^{p-2} |\nabla u_\nu^-|^2   (u_\nu^-)^{\alpha-1} \, dx\\
&+ \frac{3-p}{\sigma} \int_{\mathcal{A}(2R)} |\nabla u|^{p-2} |\nabla \varphi_R|^2  (u_\nu^-)^{\alpha+1} \varphi_R^2 \, dx\\
&+  \int_{\mathcal{A}(2R)} f'(u) (u_\nu^-)^{\alpha+1} \varphi_R^2 \, dx.
\end{split}
\end{equation}
Since $u_\nu=(\nabla u, \nu)$, where $\|\nu\|=1$, we have
\begin{equation}\label{computationdir5}
\begin{split}
&\alpha (p-1)\int_{\mathcal{A}(2R)} |\nabla u|^{p-2} |\nabla u_\nu^-|^2 (u_\nu^-)^{\alpha-1} \varphi_R^2 \, dx\\
& \leq \sigma (3-p) \int_{\mathcal{A}(2R)} |\nabla u|^{p-2} |\nabla u_\nu^-|^2   (u_\nu^-)^{\alpha-1} \, dx\\
&+ \frac{3-p}{\sigma} \int_{\mathcal{A}(2R)} |\nabla u|^{p-1} |\nabla \varphi_R|^2  (u_\nu^-)^\alpha \varphi_R^2 \, dx\\
&+  \int_{\mathcal{A}(2R)} f'(u) (u_\nu^-)^{\alpha+1} \varphi_R^2 \, dx\\
& \leq \sigma (3-p) \int_{\mathcal{A}(2R)} |\nabla u|^{p-2} |\nabla u_\nu^-|^2   (u_\nu^-)^{\alpha-1} \, dx\\
&+ \hat{C} \int_{\mathcal{A}(2R)}  |\nabla \varphi_R|(u_\nu^-)^\alpha \varphi_R^2  |\nabla \varphi_R| \, dx\\
&- L  \int_{\mathcal{A}(2R)}  (u_\nu^-)^{\alpha+1} \varphi_R^2 \, dx,
\end{split}
\end{equation}
where we used \eqref{eq:derivatafnegativa} and where $\hat{C}:={3-p}/{\sigma} \| \nabla u\|^{p-1}_\infty$. Exploiting the Young inequality with exponents $({\alpha+1})/{\alpha}$ and  $\alpha+1$ we obtain
\begin{equation}\label{computationdir6}
\begin{split}
&\alpha (p-1)\int_{\mathcal{A}(2R)} |\nabla u|^{p-2} |\nabla u_\nu^-|^2 (u_\nu^-)^{\alpha-1} \varphi_R^2 \, dx\\
& \leq \sigma (3-p) \int_{\mathcal{A}(2R)} |\nabla u|^{p-2} |\nabla u_\nu^-|^2   (u_\nu^-)^{\alpha-1} \, dx\\
&+ \frac{\hat{C}}{\alpha+1}\int_{\mathcal{A}(2R)} |\nabla \varphi_R|^{\alpha+1} \, dx \\
&+ \frac{\hat{C}(\alpha+1)}{\alpha} \int_{\mathcal{A}(2R)} |\nabla \varphi_R|^{\frac{\alpha+1}{\alpha}} (u_\nu^-)^{\alpha+1} \varphi_R^{2 \frac{\alpha+1}{\alpha}} \, dx\\
&- L  \int_{\mathcal{A}(2R)}  (u_\nu^-)^{\alpha+1} \varphi_R^2 \, dx,
\end{split}
\end{equation}
Since $|\nabla \varphi_R| \leq 2/R$ in $B_{2R} \setminus B_R$, $0 \leq \varphi_R \leq 1$ in $\R^N$ and $\varphi_R=1$ in $B_R$, we obtain
\begin{equation}\label{computationdir7}
\begin{split}
\int_{\mathcal{A}(R)} |\nabla u|^{p-2} |\nabla u_\nu^-|^2 (u_\nu^-)^{\alpha-1} \, dx &\leq \vartheta \int_{\mathcal{A}(2R)} |\nabla u|^{p-2} |\nabla u_\nu^-|^2   (u_\nu^-)^{\alpha-1} \, dx\\
& + \frac{1}{\alpha(p-1)}\left( \frac{\hat{C}(\alpha+1)}{\alpha R^\frac{\alpha+1}{\alpha}}- L \right) \int_{\mathcal{A}(2R)}  (u_\nu^-)^{\alpha+1} \varphi_R^2 \, dx,\\
&+ \frac{\bar C}{R^{\alpha-(N-1)}},
\end{split}
\end{equation}
where $\vartheta:={\sigma(3-p)}/{\alpha(p-1)}$ and $\bar C: = {2\hat{C}}/{\alpha(\alpha+1)(p-1) }$. Now we fix $\alpha>0$ such that $\alpha>N-1$, $\sigma>0$ sufficiently small such that $\vartheta < 2^{-N}$ and finally $R_0>0$ such that $\displaystyle {\hat{C}(\alpha+1)}/{\alpha R^\frac{\alpha+1}{\alpha}}- L < 0$. Having in mind all these fixed parameters let us define $$\mathcal{L}(R):= \int_{\mathcal{A}(R)} |\nabla u|^{p-2} |\nabla u_\nu^-|^2 (u_\nu^-)^{\alpha-1} \, dx.$$
It is easy to see that $\mathcal{L}(R) \leq C R^N$. By \eqref{computationdir7} we deduce that holds
$$\mathcal{L}(R) \leq \vartheta \mathcal{L}(2R) + \frac{\bar C}{R^{\alpha- (N-1)}}$$
for every $R \geq R_0$.  By applying
Lemma 2.1 in \cite{FMS} it follows that $\mathcal{L}(R)=0$ for all
$R \geq R_0$. Hence passing to the limit we obtain that $u_\nu^-=0$ in $A$.

{\em Step 2}. $u_\nu^-=0$ in $D$.\\ Let us denote by $B'$  the $(N-1)$ dimensional ball   in $\R^{N-1}$ and $\psi_R(x',x_N)=\psi_R(x')\in C_c^\infty(\mathbb R^{N-1})$ is a standard cutoff function such that
\begin{equation}\label{Eq:Cut-off1}
\begin{cases}
\psi_R \equiv 1, & \text{ in } B^{'}(0,R) \subset \mathbb{R}^{N-1},\\
\psi_R \equiv 0, & \text{ in } \mathbb{R}^{N-1} \setminus B^{'}(0,2R),\\
|\nabla \psi_R | \leq \frac 2R, & \text{ in } B^{'}(0, 2R) \setminus B^{'}(0,R) \subset  \mathbb{R}^{N-1}.
\end{cases}
\end{equation}
Let us define the cylinder $$\mathcal{C}(R):=\left\{(x',x_N)\in \mathbb R^N\,:\, \{x \in \R^N \ |  -k < x_N < k\}\cap \overline{\{B^{'}(0,R)\times \mathbb{R}\}} \right\}.$$

 We set
\begin{equation}\label{testB}
\psi:=(u_\nu^-)^\beta \psi_R^{2} \chi_{\mathcal{C}(2R)}
\end{equation}
where $\beta >1$. First of all we notice that $\psi$
belongs to $W^{1,p}_0(\mathcal{C}(2R))$ by \eqref{Eq:Cut-off1} and since $u_\nu^-=0$ on $\partial A$ (as above, see  Lemma \ref{gradBelow} and Lemma \ref{monCone}).
Recalling \eqref{linearizedA} we have also in this case that
\begin{equation}\label{linearizedEqB}
L_u(u_\nu, \psi) =0 \quad \forall \psi \in C^1_c(\R^N).
\end{equation}
Taking $\psi$ defined in \eqref{testB} in the previous equation, we obtain
\begin{equation}\label{computationdirB}
\begin{split}
&\beta \int_{\mathcal{C}(2R)} |\nabla u|^{p-2} (\nabla u_\nu, \nabla u_\nu^-) (u_\nu^-)^{\beta-1} \psi_R^2\, dx\\
& + 2 \int_{\mathcal{C}(2R)} |\nabla u|^{p-2} (\nabla u_\nu, \nabla \psi_R) (u_\nu^-)^\beta \psi_R \, dx\\
&+\beta (p-2)  \int_{\mathcal{C}(2R)} |\nabla u|^{p-4} (\nabla u, \nabla u_\nu)(\nabla u, \nabla u_\nu^-) (u_\nu^-)^{\beta-1} \psi_R^2 \, dx \\
&+ 2 (p-2)  \int_{\mathcal{C}(2R)} |\nabla u|^{p-4} (\nabla u, \nabla u_\nu)(\nabla u, \nabla \varphi_R) (u_\nu^-)^\beta \psi_R \, dx \\
&= \int_{\mathcal{C}(2R)} f'(u) (u_\nu^-)^{\beta+1} \psi_R^2 \, dx.
\end{split}
\end{equation}
Repeating verbatim the same argument  of \eqref{computationdir2}, \eqref{computationdir3} and \eqref{computationdir4}, starting by \eqref{computationdirB} we obtain
\begin{equation}\label{computationdirB2}
\begin{split}
&\beta (p-1) \int_{\mathcal{C}(2R)} |\nabla u|^{p-2} |\nabla u_\nu^-|^2 (u_\nu^-)^{\beta-1} \psi_R^2 \, dx\\
& \leq \sigma (3-p)\int_{\mathcal{C}(2R)} |\nabla u|^{p-2} |\nabla u_\nu^-|^2 \  (u_\nu^-)^{\beta-1}  \, dx\\
&+  \frac{(3-p)}{\sigma}\int_{\mathcal{C}(2R)} |\nabla u|^{p-2} |\nabla \psi_R|^2(u_\nu^-)^{\beta+1} \psi_R^2 \, dx\\
&+  \int_{\mathcal{C}(2R)}  f'(u) (u_\nu^-)^{\beta+1} \psi_R^2 \, dx.
\end{split}
\end{equation}
Since $u_\nu=(\nabla u, \nu)$  and $|\nabla u| \leq \rho$ in $\mathcal{C}(2R)$ we have
\begin{equation}\label{computationdirB3}
\begin{split}
&\int_{\mathcal{C}(2R)} |\nabla u|^{p-2} |\nabla u_\nu^-|^2 (u_\nu^-)^{\beta-1} \psi_R^2 \, dx \\
& \leq \vartheta \int_{\mathcal{C}(2R)} |\nabla u|^{p-2} |\nabla u_\nu^-|^2 \  (u_\nu^-)^{\beta-1}  \, dx\\
&+ \hat{C} \rho^{p-1} \int_{\mathcal{C}(2R)} |\nabla \psi_R|^2(u_\nu^-)^{\beta} \psi_R^2 \, dx\\
&+  C_k \int_{\mathcal{C}(2R)} (u_\nu^-)^{\beta+1} \psi_R^2 \, dx.
\end{split}
\end{equation}
where $\vartheta := {\sigma(3-p)}/{\beta(p-1)}$, $\hat{C}:={(3-p)}/{\sigma \beta (p-1)}$ and $\tilde C:={\|f'\|_{L^\infty((-1,1))}}{\beta(p-1)}$. Exploiting the Young inequality with exponents $({\beta+1})/{\beta}$ and  $\beta+1$ we obtain
\begin{equation}\label{computationdirB4}
\begin{split}
&\int_{\mathcal{C}(2R)} |\nabla u|^{p-2} |\nabla u_\nu^-|^2 (u_\nu^-)^{\beta-1} \psi_R^2 \, dx\\
& \leq \vartheta \int_{\mathcal{C}(2R)} |\nabla u|^{p-2} |\nabla u_\nu^-|^2 \  (u_\nu^-)^{\beta-1}  \, dx\\
&+ \frac{\hat{C} \rho^{p-1}}{\beta+1}\int_{\mathcal{C}(2R)} |\nabla \psi_R|^{\beta+1} \, dx \\
&+ \frac{\hat{C} \rho^{p-1}(\beta+1)}{\beta} \int_{\mathcal{C}(2R)} |\nabla \psi_R|^{\frac{\beta+1}{\beta}} (u_\nu^-)^{\beta+1} \psi_R^{2 \frac{\beta+1}{\beta}} \, dx\\
&+ \tilde C \int_{\mathcal{C}(2R)}  (u_\nu^-)^{\beta+1} \psi_R^2 \, dx.
\end{split}
\end{equation}
Since $|\nabla \psi_R| \leq 2/R$ in $B'_{2R} \setminus B'_R$, $0 \leq \psi_R \leq 1$ in $\R^N$ and $\psi_R=1$ in $B'_R$, we obtain
\begin{equation}\label{computationdirB5}
\begin{split}
&\int_{\mathcal{C}(2R)} |\nabla u|^{p-2} |\nabla u_\nu^-|^2 (u_\nu^-)^{\beta-1} \psi_R^2 \, dx \\
&\leq \vartheta \int_{\mathcal{C}(2R)} |\nabla u|^{p-2} |\nabla u_\nu^-|^2 \  (u_\nu^-)^{\beta-1}  \, dx
+ \bar{C}_{R}\int_{\mathcal{C}(2R)}  (u_\nu^-)^{\beta+1} \psi_R^2 \, dx + \frac{2^{\beta+1}\hat{C} \rho^{p-1}}{(\beta+1) R^{\beta-(N-2)}}\\
&\leq \vartheta \int_{\mathcal{C}(2R)} |\nabla u|^{p-2} |\nabla u_\nu^-|^2 \  (u_\nu^-)^{\beta-1}  \, dx+ \bar{C}_{R}\int_{B'(0,2R)} \left( \int_{-k}^{k}  \left[(u_\nu^-)^\frac{\beta+1}{2} \right]^2 \, dx_N \right) \psi_R^2(x') \, dx' \\
&+ \frac{2^{\beta+1}\hat{C} \rho^{p-1}}{(\beta+1) R^{\beta-(N-2)}}\\
&\leq \vartheta \int_{\mathcal{C}(2R)} |\nabla u|^{p-2} |\nabla u_\nu^-|^2 \  (u_\nu^-)^{\beta-1}  \, dx+ \bar{C}_{R} C_p(k)^2 \frac{(\beta+1)^2}{4} \int_{\mathcal{C}(2R)} |\partial_{x_N} u_\nu^- |^2(u_\nu^-)^{\beta-1} \psi_R^2 \, dx \\& + \frac{2^{\beta+1}\hat{C} \rho^{p-1}}{(\beta+1) R^{\beta-(N-2)}}\\
&\leq \vartheta \int_{\mathcal{C}(2R)} |\nabla u|^{p-2} |\nabla u_\nu^-|^2 \  (u_\nu^-)^{\beta-1}  \, dx\\ & + \bar{C}_{R} C_p(k)^2 \frac{(\beta+1)^2}{4} \rho^{2-p} \int_{\mathcal{C}(2R)} |\nabla u|^{p-2} |\nabla u_\nu^-|^2  (u_\nu^-)^{\beta-1} \psi_R^2 \, dx + \frac{2^{\beta+1}\hat{C} \rho^{p-1}}{(\beta+1) R^{\beta-(N-2)}},
\end{split}
\end{equation}
with $\bar{C}_{R}:={\hat{C} \rho^{p-1}(\beta+1)}/{\beta R^\frac{\beta+1}{\beta}} + \tilde C$. We point out that   in \eqref{computationdirB5} we used a Poincar\'e inequality in the set $[-k,k]$ (denoting with $C_p$ the associated constant) together with  the fact that $\psi_R=\psi_R(x')$.

Finally we choose $\beta>0$ such that $\beta>N-2$, $\vartheta>0$ sufficiently small such that $\vartheta < 2^{-N+1}$ and   $\rho>0$ sufficently small such that
\[\bar{C}_{R} C_p(k)^2 \frac{(\beta+1)^2}{2} \rho^{2-p} < 1.\]Having in mind all these fixed parameters let us define $$\mathcal{L}(R):= \int_{\mathcal{C}(R)} |\nabla u|^{p-2} |\nabla u_\nu^-|^2 (u_\nu^-)^{\beta-1} \, dx.$$
It is easy to see that $\mathcal{L}(R) \leq C R^{N-1}$. By \eqref{computationdirB5} (up to a redefining of the constant involved) we deduce that
\begin{equation}\label{eq:vittpal}
\mathcal{L}(R) \leq \vartheta \mathcal{L}(2R) + \frac{C}{R^{\beta- (N-2)}}\end{equation}
holds for every $R > 0$.  By applying
Lemma 2.1 in \cite{FMS} it follows that $\mathcal{L}(R)=0$ for all
$R >0$. Since $p<2$, passing to the limit in \eqref{eq:vittpal}, we deduce that for a.e. $x\in D$
\begin{equation}\label{eq:vittcors}\text{either } u_\nu^-(x)=0,\qquad  \text{or } |\nabla u_\nu^-(x)|=0.\end{equation}
This actually implies that $u_\nu^-(x)=0$ in $D$. Indeed let us suppose that would exist a point $P\in D$ such that $u_\nu^-(P)\neq0$. Let us consider the connected  component $\mathcal U$ of $D\setminus \{x\in D\,\,:\,\, u_\nu^-(x)=0\}$ containing $P$. By the continuity of $u_\nu^-$, it follows that  $u_\nu^-=0$ on the boundary  $\partial \mathcal U$. On the other hand $u_\nu^-$ must be constant in $\mathcal U$ (since by  \eqref{eq:vittcors} $|\nabla u_\nu^-|=0$ there) .This is a contradiction.

By this two step we deduce that $u_\nu \geq  0$ \text{in} $\R^N$.  Finally by Lemma  \ref{lem:utile} we get \eqref{eq:ultimasergioaereo}.
\end{proof}
\begin{proof}[Proof of Theorem \ref{thm:gibb}] Using Proposition \ref{monotonicitybis} we get that the solution is monotone increasing in the $y$-direction and  this implies  that  $\partial_y u\geq 0$ in $\mathbb R^N$.  In particular we have $\partial_y u> 0$ in $\mathbb R^N\setminus \mathcal{Z}_{f(u)}$ by~\eqref{mmmmmm}. By Proposition \ref{monRncone}, actually we obtain that the solution is increasing in a cone of directions close to the $y$-direction. This allows  us to show that in fact, for $i=1,2,\cdots,N-1$,  $\partial_{x_i}u=0$ in $\mathbb R^N$, just exploiting the arguments in \cite{Far99} (see also \cite{FarSciuSo, FarSoave}). We provide the details for the sake completeness. Let $\Omega$ be the set of the directions $\eta \in \mathbb{S}^{N-1}_+$ for which there exists an open neighborhood $\mathcal{O}_\eta \subset \mathbb{S}^{N-1}_+$ such that
$$\partial_\nu u = u_\nu \geq 0 \quad \text{in} \,\, \R^N \quad \text{and} \quad \partial_\nu u = u_\nu > 0\quad \text{in} \,\, \R^N \setminus \mathcal{Z}_{f(u)},$$
for every $\nu \in \mathcal{O}_\eta$. The set $\Omega$ is non-empty, since $e_N \in \Omega$, and it is also open by Proposition \ref{monRncone}. Now we want to show that it is also closed. Let $\bar \eta \in \mathbb{S}^{N-1}_+$ and let us consider the sequence $\{ \eta_n\}$ in $\Omega$ such that $\eta_n \rightarrow \bar \eta$ as $n \rightarrow +\infty$ in the topology of $\mathbb{S}^{N-1}_+$. Since by our assumptions $\partial_{\eta_n} u \geq 0$ in $\R^N$, passing to the limit we obtain that  $\partial_{\bar \eta} u \geq 0$ in $\R^N$. By Lemma \ref{lem:utile} it follows that  $\partial_{\bar \eta} u > 0$ in $\R^N \setminus \mathcal{Z}_{f(u)}$. By Proposition \ref{monRncone} there exists an open neighborhood $\mathcal{O}_{\bar \eta}$ such that \eqref{eq:ultimasergioaereo} is true for every $\nu \in \mathcal{O}_{\bar \eta}$; hence $\bar \eta \in \Omega$ and this implies that $\Omega$ is also closed. Now, since $ \mathbb{S}_+^{N-1}$ is a path-connected set, we have that $\Omega=\mathbb{S}_+^{N-1}$.
Then there exists $v\in C^{1,\alpha}_{loc}(\mathbb R)$ such that $u(x',y)=v(y)$. Now, let us assume that there exists $b \in \mathcal{Z}_{f(u)}$ {$ \setminus \{ -1,1 \} $} such that {$v'(b)=0$. Then, by $u_y \geq 0$, the level set $\{v=v(b)\}$ must be a bounded closed interval (possibly reduced to a single point),} i.e., there exist $\alpha, \beta \in \R$ with $\alpha \leq \beta$
such that
$$\{v=v(b)\}=[\alpha, \beta].$$
Therefore, by H\"opf's Lemma, we have that $v'(\beta)>0$. The latter clearly implies that $\{v=v(b)\}=\{\beta\} = \{b\}$ and so $v'(b)>0$, which is in contradiction with our initial assumption. Hence we deduce that  $\partial_y u>0$ in $\mathbb R^N$, concluding the proof.
\end{proof}

\end{document}